\documentclass[12pt]{article}
\usepackage{amssymb,amsmath, amsthm,latexsym, verbatim, amscd}
\usepackage{here}

  \newtheorem{theorem}{Theorem}[section] %
  \newtheorem{proposition}[theorem]{Proposition} %
  \newtheorem{lemma}[theorem]{Lemma} %
  
  \newtheorem{corollary}[theorem]{Corollary} %
  
  \newtheorem{definition-theorem}[theorem]{Definition-Theorem}
\theoremstyle{definition} %
  \newtheorem{definition}[theorem]{Definition} %
  \newtheorem{example}[theorem]{Example} %

  \newtheorem{question}[theorem]{Question}
  
  \newtheorem{fact}[theorem]{Fact}

  \newtheorem{remark}[theorem]{Remark} %

\newcommand{\invHom}[3]{\operatorname{Hom}_{#1}(#2,#3)}

\numberwithin{equation}{section}
\numberwithin{equation}{section}
\numberwithin{table}{section}

\begin{document}
\title{Bounded multiplicity branching
\\
 for symmetric pairs
\\[2ex]
\textit{\normalsize
Dedicated to 
 Dr.\ Karl H.\ Hofmann
 with admiration 
 and heartfelt gratitude 
 for his contributions
 to mathematics
 and his devotion
 to the community
}
}
\author{Toshiyuki KOBAYASHI
\footnote{
Graduate School of Mathematical Sciences, 
The University of Tokyo, 
Japan.  
}}
\date{}

\maketitle

\setcounter{section}{0}

\begin{abstract}
We prove that any simply connected non-compact semisimple Lie group $G$ admits
 an infinite-dimensional irreducible representation $\Pi$ 
 with bounded multiplicity property of the restriction $\Pi|_{G'}$
 for {\it{all}} symmetric pairs $(G, G')$.
We also discuss which irreducible representations $\Pi$ satisfy
 the bounded multiplicity property.  
\end{abstract}



\section{Introduction}

We initiated in \cite{xkInvent94, K95} the general study
 of {\it{multiplicities}}
 in the branching problem of reductive groups, 
 and this article is a continuation
 of the work \cite{K14, K21, K22PJA, K22, tkVarnaMin, xktoshima}.  
The goal is to prove the following theorem.  
\begin{theorem}
\label{thm:mbrest}
Any simply connected non-compact semisimple Lie group $G$ admits
 an infinite-dimensional irreducible representation $\Pi$ 
 with the bounded multiplicity property of the restriction $\Pi|_{G'}$
 for all symmetric pairs $(G, G')$:
\begin{equation}
\label{eqn:bddrest}
 \underset{\pi \in \operatorname{Irr}(G')}\sup 
 [\Pi|_{G'}:\pi]<\infty.
\end{equation} 
\end{theorem}

Let us explain some terminologies.  
We denote by $\operatorname{Irr}(G)$
 the set of irreducible objects
 in the category ${\mathcal{M}}(G)$
 of smooth admissible representations
 of a real reductive Lie group $G$
 of finite length with moderate growth, 
 which are defined on Fr{\'e}chet topological vector spaces
 \cite[Chap.~11]{WaI}.

Suppose that $G'$ is a reductive subgroup of $G$.  
For $\Pi \in {\mathcal{M}}(G)$, 
 the {\it{multiplicity}}
 of $\pi \in \operatorname{Irr}(G')$
 in the restriction $\Pi|_{G'}$ is defined by 
\begin{equation}
\label{eqn:dfn-mult}
 [\Pi|_{G'}:\pi]:=\dim_{\mathbb{C}} \invHom{G'}{\Pi|_{G'}}{\pi}
 \in {\mathbb{N}} \cup \{\infty\}, 
\end{equation}
where 
$\invHom{G'}{\Pi|_{G'}}{\pi}$ denotes the space 
 of {\it{symmetry breaking operators}}, 
 {\it{i.e.}},  
 continuous $G'$-homomorphisms
 between the Fr{\'e}chet representations.

By a {\it{symmetric pair}} $(G,G')$, 
 we mean 
 that $G'$ is an open subgroup of the fixed point group $G^{\sigma}$
 of an involutive automorphism $\sigma$
 of $G$.  
The Riemannian symmetric pair $(G,K)$
 with $\sigma$ being a Cartan involution $\theta$ and 
the {\it{group manifold case}}
 $({}^{\backprime}G \times {}^{\backprime}G, \operatorname{diag}{}^{\backprime}G)$
 are typical examples.  
The pair $(SL(n,{\mathbb{R}}), SO(p,q))$
 with $p+q=n$ is another example of symmetric pairs.  
The infinitesimal classification of irreducible symmetric pairs
 was accomplished by Berger \cite{Be57}.

Theorem \ref{thm:mbrest} may look quite surprising, 
 in view of the theorem \cite{K14}
 revealing that for \lq\lq{many}\rq\rq\ symmetric pairs $(G,G')$
 with $G'$ non-compact 
\begin{equation}
\label{eqn:infty}
\text{
 $[\Pi|_{G'}:\pi]=\infty$
 for some $\Pi \in \operatorname{Irr}(G)$
 and $\pi \in \operatorname{Irr}(G')$.  }
\end{equation}
See \cite{xKMt} for the classification of such symmetric pairs $(G,G')$.

We refer to \cite{xKVogan2015} and \cite[Sect.\ 2]{tkVarnaMin}
 for some motivation
 and perspectives of the general branching problems
 and the role of bounded multiplicity property.

The tensor product of two representations
 is a special case
 of the restriction
 with respect to a symmetric pair
 $(G \times G, \operatorname{diag}G)$.  
We also prove:

\begin{theorem}
\label{thm:tensor}
For any simply connected, non-compact semisimple Lie group $G$, 
 there exist infinite-dimensional representations $\Pi_1, \Pi_2 \in \operatorname{Irr}(G)$
 such that the tensor product representation
 has the bounded multiplicity property:
\begin{equation}
\label{eqn:bddtensor}
\underset{\Pi \in \operatorname{Irr}(G)}\sup 
  [{\Pi_1 \otimes \Pi_2}:{\Pi}]<\infty.  
\end{equation}
\end{theorem}

When $\Pi$ is a unitary representation of $G$, 
 the restriction $\Pi|_{G'}$ decomposes
 into the direct integral of irreducible unitary representations
 of the subgroup $G'$:
\begin{equation}
\label{eqn:branch}
\Pi|_{G'} \simeq \int_{\widehat{G'}}^{\oplus} m_{\Pi}(\pi) \pi d \mu(\pi), 
\end{equation}
where $\widehat{G'}$ denotes the set of equivalence classes
 of irreducible unitary representations
 ({\it{unitary dual}})
 of $G'$ equipped with Fell topology, 
 $\mu$ is a Borel measure on $\widehat{G'}$, 
 and $m_{\Pi} \colon \widehat{G'} \to {\mathbb{N}}\cup \{\infty\}$
 is a measurable function.  
The irreducible decomposition \eqref{eqn:branch} is called the {\it{branching law}}
 of the restriction $\Pi|_{G'}$.  
By the theory of nuclear spaces, 
 the multiplicity in the category of admissible representations of moderate growth dominates the one
 in the category of unitary representations, 
 namely, 
 one has the inequality
\begin{equation}
\label{eqn:nuclear}
m_{\Pi}(\pi) \le [\Pi^{\infty}|_{G'}:\pi^{\infty}]
\quad
\text{a.e. $\pi \in \widehat {G'}$ with respect to $\mu$, }
\end{equation}
 where $\Pi^{\infty} \in \operatorname{Irr}(G)$
 and $\pi^{\infty} \in \operatorname{Irr}(G')$
 denote the Fr{\'e}chet representations
 of smooth vectors of $\Pi \in \widehat G$ and $\pi \in \widehat{G'}$, 
 respectively.  
Since we can take unitarizable representations
 in Theorems \ref{thm:mbrest} and \ref{thm:tensor}
 as the proof in Sections \ref{sec:4}--\ref{sec:7} below shows, 
 one has the following:

\begin{corollary}
\label{cor:unitary}
{\rm{(1)}}\enspace
Any simply connected non-compact semisimple Lie group $G$ admits 
 an infinite-dimensional irreducible unitary representation $\Pi$
 such that the branching law \eqref{eqn:branch} of the restriction 
 $\Pi|_{G'}$ satisfies the bounded multiplicity property
 for all symmetric pairs $(G,G')$:
there exists $C>0$ such that 
\[
  m_{\Pi}(\pi)\le C
\quad
\text{a.e. $\pi \in \widehat {G'}$ with respect to $\mu$.}
\]
\newline
{\rm{(2)}}\enspace
For any simple connected, non-compact semisimple Lie group $G$, 
 there exist infinite-dimensional irreducible unitary representations $\Pi_1$ and $\Pi_2$ 
 such that the tensor product representation decomposes into the direct integral\[
  \Pi_1 \otimes \Pi_2 \simeq \int_{\widehat G}^{\oplus}m_{\Pi_1, \Pi_2}(\Pi) d \mu(\Pi), 
\]
with the bounded multiplicity property:
there exists $C>0$ such that 
\[
  m_{\Pi_1, \Pi_2}(\Pi) \le C
\quad
\text{a.e.\ $\Pi \in \widehat{G}$ with respect to $\mu$.}
\]
\end{corollary}

These results concern
 with the {\it{restriction}}.  
On the other hand, 
 the multiplicity occurring 
 in the {\it{induction}} 
$\operatorname{Ind}_{G'}^G({\bf{1}}) \simeq C^{\infty}(G/G')$ 
 is finite for any reductive symmetric pair $(G,G')$, 
where $\bf{1}$ denotes the one-dimensional 
trivial representation of $G'$, 
 see van den Ban \cite{vdB}:
\begin{equation}
\label{eqn:fmind}
\dim_{\mathbb{C}} \invHom{G}{\Pi}{C^{\infty}(G/G')}
<\infty
\quad
\text{for every $\Pi \in \operatorname{Irr}(G)$}.  
\end{equation}
More generally, 
 it was proved in \cite[Thm.\ A]{xktoshima}
 that the finite multiplicity property
 \eqref{eqn:fmind}
 is characterized by the {\it{real sphericity}}
(see Section \ref{sec:spherical} for the definition).  
We note that any reductive symmetric space is real spherical.  
When $G/G'$ is a symmetric space, 
 one has a stronger estimate than \eqref{eqn:fmind}, 
 namely, 
 the following bounded multiplicity property holds:

\begin{equation}
\label{eqn:bddind}
\underset{\Pi \in \operatorname{Irr}(G)}{\sup} \dim_{\mathbb{C}} \invHom{G}{\Pi}{C^{\infty}(G/G')}
<\infty.  
\end{equation}

More broadly, 
 it was proved in \cite[Thm.\ B]{xktoshima}
that the bounded multiplicity property \eqref{eqn:bddind} is 
 characterized by the sphericity
 of the complexification $G_{\mathbb{C}}/G_{\mathbb{C}}'$.  
We note that $G_{\mathbb{C}}/G_{\mathbb{C}}'$ is spherical
 when $G/G'$ is a symmetric space.

By Frobenius reciprocity
 $\invHom{G}{\Pi}{C^{\infty}(G/G')} \simeq \invHom{G'}{\Pi|_{G'}}{\bf{1}}$, 
 the estimate \eqref{eqn:bddind} is equivalent to 
\[
\underset{\Pi \in \operatorname{Irr}(G)}\sup 
[\Pi|_{G'}:{\bf{1}}]<\infty, 
\]
 which may be compared with \eqref{eqn:bddrest} and \eqref{eqn:infty}.

We prove Theorems \ref{thm:mbrest} and \ref{thm:tensor}
 as well as Corollary \ref{cor:unitary}
 not merely
 as the existence theorem 
 but also by exhibiting explicitly 
 which $\Pi \in \operatorname{Irr}(G)$ satisfies
 the bounded multiplicity property \eqref{eqn:bddrest}
 for $(G,G')$
 in scope of further detailed analysis 
 ({\it{e.g., }} \lq\lq{Stages B and C}\rq\rq\
 in the branching program \cite{xKVogan2015}, 
 see Section \ref{sec:branch}).

The proof of Theorems \ref{thm:mbrest} and \ref{thm:tensor}
 is reduced to the case
 where ${\mathfrak{g}}$ is simple.  
We explore in more details 
 in the setting that $G$ satisfies one of the following:

\par\indent
$\bullet$\enspace
automorphism groups of Hermitian symmetric spaces
 (Section \ref{sec:4});
\par\indent
$\bullet$\enspace
automorphism groups of para-Hermitian symmetric spaces
 (Section \ref{sec:para});
\par\indent
$\bullet$\enspace
the complex minimal nilpotent orbit has real points
 (Section \ref{sec:min});
\par\indent
$\bullet$\enspace
the complex minimal nilpotent orbit has no real point
 (Section \ref{sec:7}).

Correspondingly, 
 we shall see the bounded multiplicity property holds
 for the restriction $\Pi|_{G'}$
 when $\Pi$ is a \lq\lq{geometric quantization}\rq\rq\ 
 of certain elliptic, hyperbolic, 
 or (real) minimal nilpotent coadjoint orbits, 
 see Theorem \ref{thm:mfB}, 
 Corollary \ref{cor:para}, 
 and Theorems \ref{thm:Joseph}
 and \ref{thm:mgbdd}, 
 respectively.

The paper is organized as follows.  
Section \ref{sec:2} 
 explains some basic notions
 and known results 
 as preliminaries, 
 and Sections \ref{sec:4}--\ref{sec:7}
 provide a family of irreducible representations $\Pi$ of $G$
 that satisfy the bounded multiplicity property \eqref{eqn:bddrest}
 of the restriction of $\Pi$.  
Theorems \ref{thm:mbrest} and \ref{thm:tensor}
  will be proved in Section \ref{sec:min}
 except for ${\mathfrak{g}}={\mathfrak{s p}}(p,q)$
 or ${\mathfrak{f}}_{4(-20)}$, 
 which will be treated
 in Section \ref{sec:7}.

\section{Preliminaries}
\label{sec:2}

In this section, 
 we explain some background, basic notions, 
 and known theorems
 in proving our main results.  

\subsection{Branching problems}
\label{sec:branch}

By branching problems in representation theory, 
 we mean the broad problem of understanding
 how irreducible representations
 of a group behave when restricted to a subgroup.  
As viewed in \cite{xKVogan2015}, 
 we may divide the branching problems
 into the following three stages:
\par\noindent
{\bf{Stage A.}}\enspace
Abstract features of the restriction;
\par\noindent
{\bf{Stage B.}}\enspace
Branching law;
\par\noindent
{\bf{Stage C.}}\enspace
Construction of symmetry breaking operators.  

The role of Stage A is to develop 
 a theory on the restriction of representations
 as generally
 as possible.  
In turn, 
 we may expect a detailed study of the restriction
 in Stages B (decomposition of representations)
 and C (decomposition of vectors)
 in the \lq\lq{promising}\rq\rq\ settings 
 that are suggested by the general theory 
 in Stage A.

Theorems \ref{thm:mbrest} and \ref{thm:tensor} answer a question
 in Stage A
 of branching problems.  
In turn, 
 we may expect a detailed analysis on the restriction $\Pi|_{G'}$
 in Stages B and C.    
See \cite{Clerc, K21, KOr, KOP, KS15, xksbonvec}
{\it{e.g.}}, 
 for some recent developments
 in Stage C
 in the setting 
 where the bounded multiplicity \eqref{eqn:bddrest} holds.

\subsection{Harish-Chandra's admissibility theorem}
\label{subsec:HC}

Harish-Chandra's admissibility theorem plays
 a fundamental role
 in the algebraic study
 of representations
 of real reductive linear Lie groups $G$, 
 which guarantees a finiteness property of multiplicities
 for the restriction $G \downarrow G'$
 if $G'$ is a maximal compact subgroup $K$ of $G$.  
That is, 
 one has the following:

\begin{fact}
[{\cite[Thm.\ 3.4.10]{WaI}}]
\label{fact:HC}
Let $G'=K$.  
For any irreducible unitary representation $\Pi$ of $G$, 
 one has
\begin{equation}
\label{eqn:fweak}
     [\Pi|_{G'}: \pi]<\infty
     \quad
     \text{for all $\pi \in \operatorname{Irr}(G')$}.  
\end{equation}
\end{fact}

We explain two directions for generalizations
 of Fact \ref{fact:HC}.

One is to highlight
 {\it{$G'$-admissible restriction}} (Definition \ref{def:adm}), namely, 
 {\it{discrete decomposability}}
 as well as finite multiplicity property, 
 see Fact \ref{fact:Kadm} below.  
The other direction is to focus on the finiteness property
 of the multiplicity, 
 as we shall treat
 in Fact \ref{fact:KO} (1).

\subsection{Discretely decomposable restrictions}
\label{subsec:adm}
The notion and the results of this section will be used
 in Sections \ref{subsec:pfmfB} and \ref{subsec:smallf4}
 for the proof of the bounded multiplicity results.  

\begin{definition}
[{\cite[Sect.\ 1]{xkInvent94}}]
\label{def:adm}
A unitary representation $\Pi$ of $G$
 is {\it{$G'$-admissible}}
 if the restriction $\Pi|_{G'}$ splits into a direct sum 
 of irreducible unitary representations of $G'$:
\begin{equation}
\label{eqn:hwbdd}
\Pi|_{G'} \simeq {\sum_{\pi \in \widehat{G'}}}^{\oplus} m_{\Pi}(\pi)\pi, 
\end{equation}
with multiplicity $m_{\Pi}(\pi)<\infty$
 for all $\pi \in\widehat{G'}$.  
\end{definition}

Fact \ref{fact:HC} tells us that any $\Pi \in \widehat G$ is $K$-admissible.  
We begin with the case
 where $G'$ is compact
 but is not necessarily a maximal compact subgroup $K$.  
In this case, 
 discrete decomposability is obvious
 because $G'$ is compact, 
 and the finiteness of $m_{\Pi}(\pi)$ is non-trivial.  
We review a necessary and sufficient condition 
 for \eqref{eqn:fweak}
 when $G'$ is compact.  
In the following statement, 
 we use the letter $K'$
 instead of $G'$
 to emphasize that $K'$ is compact.

\begin{fact}
[{\cite{xkAnn98, K21Kostant}}]
\label{fact:Kadm}
Suppose that $K'$ is a subgroup
 of $K$.  
Let $\Pi \in {\mathcal{M}}(G)$.  
Then the following two conditions on the triple $(G, K', \Pi)$
 are equivalent:
\par\noindent
{\rm{(i)}}\enspace
The finite multiplicity property \eqref{eqn:fweak} holds.  
\par\noindent
{\rm{(ii)}}\enspace
$\operatorname{AS}_K(\Pi) \cap C_K(K') =\{0\}$.  
\end{fact}

Here $\operatorname{AS}_K(\Pi)$ is the asymptotic $K$-support
 of $\Pi$.  
There are only finitely many possibilities
 of asymptotic $K$-supports $\operatorname{AS}_K(\Pi)$
 for $\Pi \in {\mathcal{M}}(G)$.  
The closed cone $C_K(K')$ is the momentum set
 for the Hamiltonian action 
 on the cotangent bundle $T^{\ast}(K/K')$.  
There are two proofs
 for the implication (ii) $\Rightarrow$ (i):
 by using the singularity spectrum
 (or the wave front set) of the character 
 \cite{xkAnn98}
 and by using symplectic geometry \cite{K21Kostant}.  
The proof for the implication (i) $\Rightarrow$ (ii) is given
 in \cite{K21Kostant}.

Fact \ref{fact:Kadm}  plays
 a crucial role
 in the study of discretely decomposable restriction
 with respect to {\it{non-compact}} reductive subgroups $G'$
 \cite{xkInvent94, xkAnn98, xkInvent98, K21Kostant}.

\begin{proposition}
[{\cite[Thm.\ 4.5]{xKVogan2015}}]
\label{prop:adm}
Let $G \supset G'$ be a pair of real reductive Lie groups, 
 and $K \supset K'$ maximal compact subgroups
 modulo centers.  
For an irreducible unitary representation $\Pi$ of $G$, 
 we denote by $\Pi^{\infty} \in \operatorname{Irr}(G)$
 the Fr{\'e}chet representation of smooth vectors, 
 and by $\Pi_K$ the underlying $({\mathfrak{g}}, K)$-module.  
If one of the equivalent conditions
 in Fact \ref{fact:Kadm} holds, 
 then the restriction $\Pi|_{G'}$ is $G'$-admissible.  

Moreover, 
 the multiplicity 
$
   m_{\Pi}(\pi)=\dim_{\mathbb{C}}\invHom{G'}{\pi}{\Pi|_{G'}}
$ 
 of the discrete spectrum is finite, 
 and satisfies the following equalities.  
\begin{equation}
\label{eqn:multgk}
m_{\Pi}(\pi) = [\Pi^{\infty}|_{G'}:\pi^{\infty}]
=\dim_{\mathbb{C}} \invHom{{\mathfrak{g}}',K'}{\Pi_K}{\pi_{K'}}.  
\end{equation}
\end{proposition}

\begin{remark}
(1)\enspace
The bounded multiplicity property \eqref{eqn:bddrest} does not hold in general
 even for the case $G'=K$. 
We shall see in Theorem \ref{thm:mgbdd}
 that \eqref{eqn:bddrest} holds
 if $\Pi \in \operatorname{Irr}(G)$ is \lq\lq{small}\rq\rq\ in the sense
 that the Gelfand--Kirillov dimension of $\Pi$ equals
 half the dimension of a {\it{real}} minimal coadjoint orbit. 
\newline
(2)\enspace
The first equality in \eqref{eqn:multgk} is not true in general 
 when there is continuous spectrum
 in the restriction $\Pi|_{G'}$.  
\newline
(3)\enspace(\cite[Ex.\ 6.3]{mf-korea})\enspace
The multiplicity $m_{\Pi}(\pi)$ of discrete spectrum
may be infinite 
 even for reductive symmetric pairs $(G,G')$
 if one of the equivalent conditions in Fact \ref{fact:Kadm} fails.  
\end{remark}
See \cite{xdgv, decoAq, KO15} for a classification theory
 of the triple $(G,G',\Pi)$ 
 satisfying the equivalent conditions
 in Fact \ref{fact:Kadm}.

\subsection{Spherical spaces and real spherical spaces}
\label{sec:spherical}

In \cite{K14} and \cite[Thms.~C and D]{xktoshima}
 we proved the following geometric criteria
 that concern {\it{all}} $\Pi \in {\operatorname{Irr}}(G)$
 and {\it{all}} $\pi \in {\operatorname{Irr}}(G')$:
\begin{fact}
\label{fact:KO}
Let $G \supset G'$ be a pair of real reductive algebraic Lie groups.  
\par\noindent
{\rm{(1)}}\enspace
{\bf{Finite multiplicity}} for a pair $(G,G')$:\enspace
\[
  [\Pi|_{G'}:\pi]<\infty, 
\quad
 {}^{\forall}\Pi \in \operatorname{Irr}(G), 
 {}^{\forall}\pi \in \operatorname{Irr}(G')
\]
if and only if $(G \times G')/\operatorname{diag}G'$ is real spherical.  
\par\noindent
{\rm{(2)}}\enspace
{\bf{Bounded multiplicity}} for a pair $(G,G')$: \enspace
\begin{equation}
\label{eqn:BB}
  \underset{\Pi \in \operatorname{Irr}(G)}\sup\,\,
  \underset{\pi \in \operatorname{Irr}(G')}\sup\,\,
  [\Pi|_{G'}:\pi]<\infty
\end{equation}
if and only if $(G_{\mathbb{C}} \times G_{\mathbb{C}}')/\operatorname{diag} G_{\mathbb{C}}'$ is spherical.   

\end{fact}

Here a complex $G_{\mathbb{C}}$-manifold $X$
 is called {\it{spherical}}
 if a Borel subgroup of $G_{\mathbb{C}}$
 has an open orbit in $X$, 
 and that a $G$-manifold $Y$ is called
 {\it{real spherical}} (\cite{K95})
 if a minimal parabolic subgroup of $G$
 has an open orbit in $Y$.

A remarkable discovery in \cite{xktoshima} includes 
 that the bounded multiplicity property \eqref{eqn:BB} is determined
 only by the complexified Lie algebras
 ${\mathfrak{g}}_{\mathbb{C}}$ and ${\mathfrak{g}}_{\mathbb{C}}'$.  
In particular, 
 the classification
 of such pairs $(G,G')$
 is quite simple, 
 because it is reduced to a classical result 
 when $G$ is compact \cite{xkramer}:
the pair $({\mathfrak{g}}_{\mathbb{C}}, {\mathfrak{g}}_{\mathbb{C}}')$ is the direct sum of the following ones
up to abelian ideals:
\begin{equation}
\label{eqn:BBlist}
({\mathfrak{sl}}_n, {\mathfrak{gl}}_{n-1}), 
({\mathfrak{so}}_{n}, {\mathfrak{so}}_{n-1}),
\text{ or } 
({\mathfrak{so}}_8, {\mathfrak{spin}}_7). 
\end{equation}

On the other hand, 
 the finite multiplicity property in Fact \ref{fact:KO} (1)
 depends on real forms $G$ and $G'$.  
For instance, 
 it is fulfilled 
 for any Riemannian symmetric pair $(G,K)$
 because the Iwasawa decomposition tells us
 that $(G \times G')/\operatorname{diag}G'$
 is real spherical
 if $G'=K$, 
 whereas the finiteness of the $K$-multiplicity traces back 
 to Harish-Chandra's admissibility theorem 
(Fact \ref{fact:HC}).  
(Actually, 
 \cite{xktoshima} in this specific case gives a proof 
 that a quasi-simple irreducible representation of $G$ is $K$-admissible
 by using the boundary value problem
 of a system of partial differential equations.)

\subsection{Visible actions on complex manifolds}

Suppose a (real) Lie group $G$ acts holomorphically 
 on a connected complex manifold $D$.  

\begin{definition}
[{\cite[Def.\ 3.3.1]{xrims40}}]
\label{def:visible}
The action is called {\it{strongly visible}}
 if there exist a non-empty $G$-invariant open subset $D'$ of $D$, 
 a totally real submanifold $S$, 
and an anti-holomorphic diffeomorphism $\sigma$ of $D'$
 such that
\begin{equation*}
\text{$D'=G \cdot S$, }
\text{$\sigma|_S=\operatorname{id}$, and }
\text{$\sigma$ preserves each $G$-orbit in $D'$.  }
\end{equation*}
\end{definition}

Loosely speaking, 
the significance  of this definition is that, 
 for any $G$-equivariant holomorphic vector bundle 
 ${\mathcal{V}} \to D$
 on which $G$ acts strongly visibly on $D$, 
 the multiplicity-free property propagates from 
 fibers to sections, 
 see \cite[Thm.\ 4]{xrims40}
 for a rigorous formulation.

We shall utilize the following results
 in Sections \ref{sec:4} and \ref{sec:para}.  
\begin{fact}
[{\cite[Thm.\ 1.5]{visibleGK}}]
\label{fact:GKvisible}
Let $G/K$ be a Hermitian symmetric space, 
 either of compact type or of non-compact type.  
Then the $G'$-action on $G/K$
 is strongly visible
 for any symmetric pair $(G,G')$.  
\end{fact} 

\subsection{Coisotropic action on coadjoint orbits}
\label{subsec:coiso}

Let $V$ be a vector space
 equipped with a symplectic form $\omega$.  
A subspace $W$ is called {\it{coisotropic}}
 if 
$
\{v \in V: \text{$\omega(v,\cdot)$ vanishes on $W$}\}
$
 is contained in $W$.

The concept of coisotropic actions is defined infinitesimally as follows.  
\begin{definition}
[Huckleberry--Wurzbacher {\cite{huwu90}}]
\label{def:coisotropic}
Let $H$ be a connected Lie group, 
 and $X$ a Hamiltonian $H$-manifold.  
The $H$-action is called
 {\it{coisotropic}}
 if there is an $H$-stable open dense subset $U$ of $X$
 such that $T_x(H \cdot x)$ is a coisotropic subspace
 in the tangent space $T_x X$ for all $x \in U$.  
\end{definition}

Suppose that ${\mathbb{O}}$ is a coadjoint orbit
 of a connected Lie group $G$
 through $\lambda \in {\mathfrak{g}}^{\ast}$.  
Denote by $G_{\lambda}$ the stabilizer subgroup of $\lambda$ in $G$, 
 and by ${\mathfrak{Z}}_{\mathfrak{g}}(\lambda)$ its Lie algebra.  
The Kirillov--Kostant--Souriau symplectic form $\omega$
 on ${\mathbb{O}} \simeq G/G_{\lambda}$
 is given at the tangent space
 $T_{\lambda} {\mathbb{O}} \simeq {\mathfrak{g}}/{\mathfrak{Z}}_{\mathfrak{g}}(\lambda)$
 by
\begin{equation*}
\omega \colon 
{\mathfrak{g}}/{\mathfrak{Z}}_{\mathfrak{g}}(\lambda) \times {\mathfrak{g}}/{\mathfrak{Z}}_{\mathfrak{g}}(\lambda)
\to
{\mathbb{R}}, 
\quad
(X,Y) \mapsto \lambda([X,Y]).  
\end{equation*}

Suppose $G$ is semisimple.  
Then the Killing form induces 
 an isomorphism 
 ${\mathfrak{g}}^{\ast} \overset \sim \to {\mathfrak{g}}$, 
 $\lambda \mapsto X_{\lambda}$.  
The following result is useful in later argument.  
\begin{lemma}
[{\cite[Lem.\ 2]{tkVarnaMin}}]
\label{lem:1.7}
Let $H$ be a connected subgroup 
 with Lie algebra ${\mathfrak{h}}$.  
The $H$-action on a coadjoint orbit ${\mathbb{O}}$
 is coisotropic 
 if there exists a subset $S$ (slice)
 in ${\mathbb{O}}$
 with the following two properties:
\begin{align}
&\text{$\operatorname{Ad}^{\ast}(H) S$ is open dense in ${\mathbb{O}}$}, 
\notag
\\
\label{eqn:coiso}
&\text{$({\mathfrak{h}}+{\mathfrak{Z}}_{\mathfrak{g}}(\lambda))^{\perp}
\subset [X_{\lambda}, {\mathfrak{h}}]$
\qquad 
 for any $\lambda \in S$.}
\end{align}
Here $\perp$ stands for the orthogonal subspace
 with respect to the Killing form.  
\end{lemma}

The original proof of Fact \ref{fact:KO} in \cite{xktoshima} utilized
 hyperfunction boundary maps 
 for the \lq\lq{if}\rq\rq\ part
 ({\it{i.e.,}} the sufficiency of the finite multiplicity property)
 and a generalized Poisson transform \cite{K14}
 for the \lq\lq{only if}\rq\rq\ part.  
An alternative approach in \cite{K22, Tu}
 for the proof of the 
 \lq{if}\rq\ part of Fact \ref{fact:KO} (2)
 used a theory of holonomic ${\mathcal{D}}$-modules, 
 which is also the method of Theorems \ref{thm:introQsph2} and \ref{thm:fmtensor} below.  
Our proof in this article still uses
 a theory of ${\mathcal{D}}$-modules, 
 and more precisely, 
 the following:

\begin{theorem}
[{\cite{Ki}}]
\label{thm:Ki}
Let $\operatorname{Ann}\Pi$ be the annihilator
 of $\Pi \in {\mathcal{M}}(G)$
 in the universal enveloping algebra $U({\mathfrak{g}}_{\mathbb{C}})$.  
Assume that the $G_{\mathbb{C}}'$-action
 on the associated variety ${\mathcal{V}}(\operatorname{Ann}\Pi)$
 is coisotropic.  
Then the restriction $\Pi|_{G'}$ has the bounded multiplicity property
 \eqref{eqn:bddrest}.  
\end{theorem}

The associated variety ${\mathcal{V}}(\operatorname{Ann}\Pi)$ is the closure
 of a single nilpotent coadjoint orbit
 if $\Pi \in \operatorname{Irr}(G)$ 
 \cite{BB82, Joseph85}.    
We note
 that the assumption in Theorem \ref{thm:Ki} depends
 only on the pair $({\mathfrak{g}}_{\mathbb{C}}, {\mathfrak{g}}_{\mathbb{C}}')$
 of the complexified Lie algebras 
 as in Fact \ref{fact:KO} (2).

\section{Restriction of highest weight modules}
\label{sec:4}

In this section
 we discuss the bounded multiplicity property \eqref{eqn:bddrest}
 for a symmetric pair $(G,G')$
 when $\Pi$ is a highest weight module of $G$.  
We shall see Theorem \ref{thm:mfB} implies 
 Theorems \ref{thm:mbrest} and \ref{thm:tensor}
 when $G$ is the automorphism group
 of a Hermitian symmetric space, 
see \eqref{eqn:GHerm} below
 for the list of such simple Lie algebras ${\mathfrak{g}}$.

\subsection{Preliminaries for highest weight modules}

Let $G$ be a non-compact simple Lie group,
 $\theta$ a Cartan involution of $G$,
 and $K := \{g \in G:\theta g = g\}$.
We write
 $\mathfrak{g} = \mathfrak{k} + \mathfrak{p}$ 
 for the  corresponding Cartan decomposition
 of the Lie algebra $\mathfrak{g}$ of $G$. 

We assume that $G$ is of {\it Hermitian type},
 that is,
 the Riemannian symmetric space $G/K$ carries
 the structure of a Hermitian symmetric space,
 or equivalently, 
 the center $\mathfrak{c}(\mathfrak{k})$ of $\mathfrak{k}$ is non-trivial. 
The classification of simple Lie algebras ${\mathfrak {g}}$
 of Hermitian type
 is given as follows:
\begin{equation}
\label{eqn:GHerm}
{\mathfrak {su}}(p,q)\,,\ 
{\mathfrak {sp}}(n,{\mathbb{R}})\,,\ 
{\mathfrak {so}}^{\ast}(2m)\,,\ 
{\mathfrak {so}}(m,2) \ (m \ne 2)\,,\ 
{\mathfrak {e}}_{6(-14)}\,,\ 
{\mathfrak {e}}_{7(-25)} \, .  
\end{equation}

In this case, 
 there exists a characteristic element
 $Z \in \mathfrak{c}(\mathfrak{k})$
 such that 
\begin{equation}
\label{eqn:gkppz}
      \mathfrak{g}_\mathbb{C} 
    := \mathfrak{g} \otimes \mathbb{C}
     = \mathfrak{k}_\mathbb{C} \oplus \mathfrak{p}_+ \oplus \mathfrak{p}_-
\end{equation}
 is the eigenspace decomposition of $\operatorname{ad}(Z)$
 with eigenvalues $0$, $\sqrt{-1}$ and $-\sqrt{-1}$,
 respectively, 
and that 
$\mathfrak{c}(\mathfrak{k}) = \mathbb{R} Z$.

Suppose $V$ is an irreducible $({\mathfrak {g}}_{\mathbb{C}}, K)$-module.  
We set
\begin{equation}
\label{eqn:Hpk}
   V^{\mathfrak{p}_+}
   := \{v \in V : Y v = 0 
   \ \text{ for any } Y \in \mathfrak{p}_+\}\, .
\end{equation}
Since $K$ normalizes $\mathfrak{p}_+$,
$V^{\mathfrak{p}_+}$ is a $K$-submodule.
Further,
$V^{\mathfrak{p}_+}$
 is either zero or an irreducible finite-dimensional representation of $K$.
We say
$V$ is 
 a \textit{highest weight module}
 if $V^{\mathfrak{p}_+} \neq \{0\}$, 
 and of {\it{scalar type}}
 if $\dim_{\mathbb{C}} V^{{\mathfrak{p}}_+}=1$.  
For any non-compact simple Lie group $G$
 of Hermitian type, 
there exist infinitely many irreducible unitary highest weight representations of scalar type.

For any symmetric pair $(G,G')$, 
 the $G'$-action on the Hermitian symmetric space $G/K$
 is strongly visible (Fact \ref{fact:GKvisible}).  
Correspondingly, 
 we proved in \cite[Thms.\ A and C]{mf-korea}
 the following multiplicity-free theorems.  
\begin{fact}
[multiplicity-free theorem]
\label{fact:mfA}
Let $G$ be a non-compact simple Lie group of Hermitian type, 
 and $\Pi$, $\Pi_1$, $\Pi_2$ irreducible unitary highest weight representations
 of scalar type.  
\newline
{\rm{(1)}}\enspace
The restriction $\Pi|_{G'}$ is multiplicity-free
 for any symmetric pair $(G,G')$.
\newline
{\rm{(2)}}\enspace
The tensor product $\Pi_1 \otimes \Pi_2$ is multiplicity-free.  
\end{fact}

The following theorem asserts that
the multiplicities are still uniformly bounded even if we drop
 the assumption
 that $\pi$ is of scalar type.  

\begin{theorem}
[uniformly bounded multiplicities]
\label{thm:mfB}
Let $\Pi$, $\Pi_1$, $\Pi_2$ be the smooth representations 
 of irreducible unitary highest weight representations of $G$.
\newline
{\rm{(1)}}\enspace
The restriction $\Pi|_{G'}$ satisfies bounded multiplicity property \eqref{eqn:bddrest}
 for any symmetric pair $(G,G')$.  
\newline
{\rm{(2)}}\enspace
The tensor product $\Pi_1 \otimes \Pi_2$ satisfies
 the bounded multiplicity property \eqref{eqn:bddtensor}.   
\end{theorem}

\subsection{Involutions on Hermitian symmetric spaces}
\label{subsec:1.5}  %
The branching law
in Fact \ref{fact:mfA} formulated in the category
 of unitary representations 
 may and may not contain
discrete spectra.  
To clarify this, 
 we observe
 that there are two types of involutions $\sigma$
 of a non-compact simple Lie group $G$ of Hermitian type.  
Without loss of generality, 
 we may assume
 that $\sigma$ commutes 
 with the Cartan involution $\theta$.
We use the same letter $\sigma$ to denote
 its differential.
Then $\sigma$ stabilizes $\mathfrak{k}$ and also $\mathfrak{c}(\mathfrak{k})$.
Because $\sigma^2 = \operatorname{id}$
 and $\mathfrak{c}(\mathfrak{k}) = \mathbb{R} Z$,
 there are two possibilities:
\begin{align}
    \sigma Z &= Z \, ,
\label{eqn:1.5.1}
\\
   \sigma Z &= -Z \, .
\label{eqn:1.5.2}
\end{align}
Geometrically, 
 the condition \eqref{eqn:1.5.1} %
 implies: 
\newline\indent{1-a)}\enspace
 $\sigma$ acts {\bf holomorphically} 
 on the Hermitian symmetric space $G/K$,
\newline\indent{1-b)}\enspace
 $G^\sigma/K^\sigma \hookrightarrow G/K$ defines a complex submanifold,
\newline
whereas the condition
\eqref{eqn:1.5.2} %
implies:
\newline\indent{2-a)}\enspace
 $\sigma$ acts {\bf anti-holomorphically} on  $G/K$,
\newline\indent{2-b)}\enspace
 $G^\sigma/K^\sigma \hookrightarrow G/K$ defines a totally real submanifold.

\begin{definition}
\label{def:holo-anti}
We say the involutive automorphism $\sigma$ is 
\textit{of holomorphic type} if \eqref{eqn:1.5.1} is satisfied,
and is of 
\textit{anti-holomorphic type}
if \eqref{eqn:1.5.2} is satisfied.
The same terminology will be applied also to the symmetric pair 
$(G,G')$
(or its Lie algebras $(\mathfrak{g}, \mathfrak{g}')$)
corresponding to the involution $\sigma$.
\end{definition}

The restriction $\Pi|_{G'}$ is discretely decomposable
 if $(G,G')$ is of holomorphic type
 for any unitary highest weight representation $\Pi$ of $G$
 (\cite{xkAnn98} or \cite[Thm.\ 7.4]{JFA98}).  

\subsection{Proof of Theorem \ref{thm:mfB}}
\label{subsec:pfmfB}
The bounded multiplicity property 
 for symmetric pairs $({\mathfrak{g}},{\mathfrak{g}}')$
 of holomorphic type was established 
 in \cite[Thm.\ B]{mf-korea}
 in the category of unitary representations.  
Since Theorem \ref{thm:mfB} is formulated in the category 
of smooth admissible representations, 
we need some additional argument.

\begin{proof}
[Proof of Theorem \ref{thm:mfB}]
First, 
 suppose that the symmetric pair $(G,G')$ is of holomorphic type.
In this case, 
 any irreducible highest weight module
 of $G$ 
 is $K'$-admissible, 
 hence $G'$-admissible
 (Definition \ref{def:adm}), 
 see \cite{xkAnn98} or \cite[Thm.\ 7.4]{JFA98}.
In turn, 
 the bounded multiplicity theorem
 (\cite[Thm.\ B]{mf-korea})
 in the category of unitary representations
 implies  the one in the category
 of smooth admissible representations by Proposition \ref{prop:adm}.

Next suppose that $(G,G')$ is of anti-holomorphic type.  
Via the identification ${\mathfrak{g}}^{\ast}\simeq {\mathfrak{g}}$, 
the associated variety 
 is the closure of an adjoint orbit $\operatorname{Ad}(G_{\mathbb{C}}) X$
 for some $X \in {\mathfrak{p}}_+$.  
Then Theorem \ref{thm:mfB} reduces
 to the following geometric results owing to Theorem \ref{thm:Ki}.  
\end{proof}

\begin{theorem}
\label{thm:220928}
Let $G$ be a non-compact simple Lie group
 of Hermitian type.  
Retain the notation as in \eqref{eqn:gkppz}.   
\newline
{\rm{(1)}}\enspace
If $\sigma$ is of anti-holomorphic type, 
 then the $G_{\mathbb{C}}^{\sigma}$-action on $\operatorname{Ad}(G_{\mathbb{C}})X$
 is coisotropic for any $X \in {\mathfrak{p}}_+$.  
\newline
{\rm{(2)}}\enspace
The diagonal $G_{\mathbb{C}}$-action
 on $\operatorname{Ad}(G_{\mathbb{C}}) X \times \operatorname{Ad}(G_{\mathbb{C}}) Y$
is  coisotropic 
 for any $X \in {\mathfrak{p}}_+$ and $Y \in {\mathfrak{p}}_-$.  
\end{theorem}

\begin{proof}
(1)\enspace
For any non-zero $X \in {\mathfrak{p}}_+$, 
 one can take $Y \in {\mathfrak{p}}_-$
 and $H \in {\mathfrak{k}}_{\mathbb{C}}$
 such that $\{X, H, Y\}$
 forms an ${\mathfrak{s l}}_2$-triple.  
We write ${\mathfrak{s l}}_2^X$ 
 for the corresponding complex subalgebra
 in ${\mathfrak{g}}_{\mathbb{C}}$.

Since $\sigma Z=-Z$, 
 one has $\sigma {\mathfrak{p}}_+={\mathfrak{p}}_-$.  
Moreover, 
 one has ${\mathfrak{Z}}_{\mathfrak{g}_{\mathbb{C}}}(X) \supset {\mathfrak{p}}_+$
 because ${\mathfrak{p}}_+$ is abelian.  
Hence the decomposition \eqref{eqn:gkppz} yields 
\begin{align}
{\mathfrak{g}}_{\mathbb{C}}
=&\sigma{\mathfrak{p}}_+ + \mathfrak{k}_{\mathbb{C}} + {\mathfrak{p}}_+
\notag
\\
=&\sigma ({\mathfrak{Z}}_{\mathfrak{g}_{\mathbb{C}}}(X))+{\mathfrak{k}}_{\mathbb{C}}+{\mathfrak{Z}}_{\mathfrak{g}_{\mathbb{C}}}(X)
\notag
\\
=&{\mathfrak{g}}_{\mathbb{C}}^{\sigma}+ \mathfrak{k}_{\mathbb{C}} +{\mathfrak{Z}}_{\mathfrak{g}_{\mathbb{C}}}(X).  
\label{eqn:gskz}
\end{align}
We set $S:=\operatorname{Ad}(K_{\mathbb{C}}) X$.  
The equality \eqref{eqn:gskz} implies that 
$\operatorname{Ad}(G_{\mathbb{C}}^{\sigma})S$ is open 
in $\operatorname{Ad}(G_{\mathbb{C}}) X$.  
By Lemma \ref{lem:1.7}, 
 it suffices to show
\[ 
({\mathfrak{g}}_{\mathbb{C}}^{\sigma} + {\mathfrak{Z}}_{\mathfrak{g}_{\mathbb{C}}}(W))^{\perp}
 \subset [W, {\mathfrak{g}}_{\mathbb{C}}^{\sigma}]
\quad\text{for all $W \in S$}.  
\]
Without loss of generality, 
 we may replace
 $W=\operatorname{Ad}(k)X$ $(\in {\mathfrak{p}}_+)$ with $X$.

We claim the following equality
\begin{equation}
\label{eqn:22102133}
   [X, {\mathfrak{p}}_-]
=
   {\mathfrak{Z}}_{\mathfrak{k}_{\mathbb{C}}}(X)^{\perp}
\quad
\text{in ${\mathfrak{k}}_{\mathbb{C}}$}, 
\end{equation}
where the right-hand side stands for the orthogonal complement
 of ${\mathfrak{Z}}_{\mathfrak{k}_{\mathbb{C}}}(X)$ in ${\mathfrak{k}}_{\mathbb{C}}$
 with respect to the Killing form $B$ of ${\mathfrak{g}}_{\mathbb{C}}$.  
The inclusion 
$
   [X, {\mathfrak{p}}_-] 
   \subset
  {\mathfrak{Z}}_{\mathfrak{k}_{\mathbb{C}}}(X)^{\perp}
$
 is direct
 because 
$
    B([X, {\mathfrak{p}}_-], W)
  = B([X,W], {\mathfrak{p}}_-)
  = \{0\}
$
 for any $W \in {\mathfrak{Z}}_{\mathfrak{k}_{\mathbb{C}}}(X)$.  
On the other hand, 
 since $\dim\operatorname{Ad}(G_{\mathbb{C}})X=2 \dim\operatorname{Ad}(K_{\mathbb{C}})X$, 
 one has 
 $\dim [X, {\mathfrak{p}}_{\mathbb{C}}]
 =\dim {\mathfrak{k}}_{\mathbb{C}}- \dim {\mathfrak{Z}}_{\mathfrak{k}_{\mathbb{C}}}(X)$.  
As $X$ is an element
 of the abelian subalgebra ${\mathfrak{p}}_+$, 
 one has 
 $[X, {\mathfrak{p}}_-]=[X, {\mathfrak{p}}_{\mathbb{C}}]$, 
 and thus the equality \eqref{eqn:22102133} is proved.

Since ${\mathfrak{p}}_- = \sigma({\mathfrak{p}}_+)$
 and ${\mathfrak{p}}_+ \subset {\mathfrak{Z}}_{\mathfrak{p}_{\mathbb{C}}}(X)$, 
 one has 
\[
 {\mathfrak{p}}_{\mathbb{C}} = {\mathfrak{p}}_- \oplus {\mathfrak{p}}_+
 = {\mathfrak{p}}_{\mathbb{C}}^{\sigma} + {\mathfrak{Z}}_{\mathfrak{p}_{\mathbb{C}}}(X), 
\]
hence 
$
   {\mathfrak{g}}_{\mathbb{C}}^{\sigma} + {\mathfrak{Z}}_{\mathfrak{g}_{\mathbb{C}}}(X)
   =
  {\mathfrak{p}}_{\mathbb{C}} + {\mathfrak{k}}_{\mathbb{C}}^{\sigma} 
  + {\mathfrak{Z}}_{\mathfrak{k}_{\mathbb{C}}}(X).
$
Therefore
\begin{align}
\notag
  ({\mathfrak{g}}_{\mathbb{C}}^{\sigma} + {\mathfrak{Z}}_{\mathfrak{g}_{\mathbb{C}}}(X))^{\perp}
  =& ({\mathfrak{k}}_{\mathbb{C}}^{\sigma} + {\mathfrak{Z}}_{\mathfrak{k}_{\mathbb{C}}}(X))^{\perp}
\quad
\text{in ${\mathfrak{k}}_{\mathbb{C}}$}
\\
\label{eqn:22102130}
 =&[X, {\mathfrak{p}}_-]^{-\sigma}.  
\end{align} 
Since $[X, {\mathfrak{p}}_-]=\{[X, W + \sigma W]: W \in {\mathfrak{p}}_-\}
=[X,{\mathfrak{p}}_{\mathbb{C}}^{\sigma}]$, 
 we have shown
 the desired inclusive relation 
$
   ({\mathfrak{g}}_{\mathbb{C}}^{\sigma} + {\mathfrak{Z}}_{\mathfrak{g}_{\mathbb{C}}}(X))^{\perp}
\subset
[X, {\mathfrak{g}}_{\mathbb{C}}^{\sigma}].  
$
\newline
(2)\enspace
We apply the same argument as in (1)
 and obtain 
\begin{align*}
{\mathfrak{g}}_{\mathbb{C}} \oplus {\mathfrak{g}}_{\mathbb{C}}
=&\operatorname{diag} {\mathfrak{g}}_{\mathbb{C}}
 + ({\mathfrak{k}}_{\mathbb{C}} \oplus {\mathfrak{k}}_{\mathbb{C}})
 + ({\mathfrak{Z}}_{\mathfrak{g}_{\mathbb{C}}}(X) \oplus {\mathfrak{Z}}_{\mathfrak{g}_{\mathbb{C}}}(Y))
\\
=&\operatorname{diag} {\mathfrak{g}}_{\mathbb{C}}
 + [\operatorname{diag} {\mathfrak{g}}_{\mathbb{C}}, (X,Y)]
 + ({\mathfrak{Z}}_{\mathfrak{g}_{\mathbb{C}}}(X) \oplus {\mathfrak{Z}}_{\mathfrak{g}_{\mathbb{C}}}(Y)).  
\end{align*}
Thus, 
 by setting a submanifold
 $S:= \operatorname{Ad}(K_{\mathbb{C}} \times K_{\mathbb{C}})(X,Y)$, 
 one sees that $\operatorname{Ad}(\operatorname{diag}G_{\mathbb{C}}) S$
 is open
 in $\operatorname{Ad}(G_{\mathbb{C}} \times G_{\mathbb{C}})(X,Y)$
 and that 
\begin{equation*}
(\operatorname{diag} {\mathfrak{g}}_{\mathbb{C}}
+ ({\mathfrak{Z}}_{\mathfrak{g}_{\mathbb{C}}}(X') \oplus {\mathfrak{Z}}_{\mathfrak{g}_{\mathbb{C}}}(Y'))^{\perp}
\subset 
[(X',Y'), \operatorname{diag} {\mathfrak{g}}_{\mathbb{C}}]
\end{equation*}
for any $(X',Y') \in S$.  
Now the second assertion follows from Lemma \ref{lem:1.7}.  
\end{proof}

\section{Degenerate principal series representations}
\label{sec:para}

In this section 
 we discuss which degenerate principal series representation $\Pi$ of $G$
 satisfies the bounded multiplicity property \eqref{eqn:bddrest}
 for a symmetric pair $(G,G')$.  
In particular, 
 we shall see 
 that Theorems \ref{thm:mbrest} and \ref{thm:tensor} hold
 if $G$ is the automorphism group
 of a para-Hermitian symmetric space, 
 see Table \ref{table:4.1} below
 for the list of such simple Lie algebras ${\mathfrak{g}}$.

\subsection{Bounded multiplicity theorems
 for the restriction of degenerate principal series representations}

For a reductive Lie group $G$, 
 we write $G_U$
 for the compact real form
 of the complex Lie group $G_{\mathbb{C}}$
 with Lie algebra ${\mathfrak{g}}_{\mathbb{C}}=\operatorname{Lie}(G) \otimes_{\mathbb{R}} {\mathbb{C}}$.  

For a Lie group $P$, 
 we write $\operatorname{Char}(P)$
 for the set of the equivalence classes
 of one-dimensional representations of $P$, 
 and $\operatorname{Irr}(P)_f$ 
 for that of finite-dimensional irreducible representations of $P$.

The following theorems are special cases
 of the general results \cite[Thm.\ 1.4]{K22}.  

\begin{theorem}
[{\cite[Ex.\ 4.5]{K22}}]
\label{thm:introQsph2}
Let $G \supset G'$ be a pair of real reductive algebraic Lie groups, 
 and $P$ a parabolic subgroup of $G$.  
Then one has the equivalence on the triple $(G,G';P):$
\par\noindent
{\rm{(i)}}\enspace
One has
\[
\underset{\chi \in \operatorname{Char}(P)}\sup
\,\,
\underset{\pi \in \operatorname{Irr}(G')}\sup 
[\operatorname{Ind}_P^G(\chi)|_{G'}:\pi]< \infty.  
\]
\par\noindent
{\rm{(ii)}}\enspace
There exists $C>0$ such that 
\[
   \underset{\pi \in \operatorname{Irr}(G')}\sup 
   [\operatorname{Ind}_P^G(\xi)|_{G'}:\pi]< C \dim \xi
\]
for any $\xi \in \operatorname{Irr}(P)_f$.  
\par\noindent
{\rm{(iii)}}\enspace
$G_{\mathbb{C}}/P_{\mathbb{C}}$ is strongly $G_U'$-visible
 (Definition \ref{def:visible}).  
\par\noindent
{\rm{(iv)}}\enspace
$G_{\mathbb{C}}/P_{\mathbb{C}}$ is $G_{\mathbb{C}}'$-spherical.  
\end{theorem}

\begin{theorem}
[{\cite[Cor.\ 4.10]{K22}}]
\label{thm:fmtensor}
Let $G$ be a real reductive algebraic Lie group, 
 and $P_j$ $(j=1,2)$ parabolic subgroups.  
Then the following four conditions
 on the triple $(G,P_1, P_2)$
 are equivalent:
\par\noindent
{\rm{(i)}}\enspace
One has 
\begin{equation}
\label{eqn:JLT4102}
\underset{
\chi_1 \in \operatorname{Char}(P_1)
}
\sup
\,\,
\underset{
\chi_2 \in \operatorname{Char}(P_2)
}
\sup
\,\,
\underset{\Pi \in \operatorname{Irr}(G)}\sup
[{\operatorname{Ind}_{P_1}^G(\chi_1) \otimes \operatorname{Ind}_{P_2}^G(\chi_2)}:\Pi]
   <\infty.
\end{equation}

\par\noindent
{\rm{(ii)}}\enspace
There exists $C>0$ such that 
\[
\underset{
\Pi \in \operatorname{Irr}(G)
}
\sup
[{\operatorname{Ind}_{P_1}^G(\xi_1) \otimes \operatorname{Ind}_{P_2}^G(\xi_2)}:\Pi]
   \le C \dim \xi_1 \dim \xi_2
\]
for any $\xi_1, \xi_2 \in \operatorname{Irr}(P)_f$.  

\par\noindent
{\rm{(iii)}}\enspace
$(G_{\mathbb{C}} \times G_{\mathbb{C}})/({P_1}_{\mathbb{C}} \times {P_2}_{\mathbb{C}})$
 is $\operatorname{diag}(G_U)$-strongly visible.  

\par\noindent
{\rm{(iv)}}\enspace
$(G_{\mathbb{C}} \times G_{\mathbb{C}})/({P_1}_{\mathbb{C}} \times {P_2}_{\mathbb{C}})$
 is $\operatorname{diag}(G_{\mathbb{C}})$-spherical.  
\end{theorem}

\begin{remark}
\par\noindent(1)\enspace
A distinguished feature in Theorem \ref{thm:introQsph2} is 
 that the necessary and sufficient condition 
of the bounded multiplicity property
 is given only by the triple $({\mathfrak{g}}_{\mathbb{C}}, {\mathfrak{g}}_{\mathbb{C}}', {\mathfrak{p}}_{\mathbb{C}})$
 of complexified Lie algebras, 
 which traces back
 to \cite{K95, xktoshima}.  
\par\noindent(2)\enspace
For each complex symmetric pair $(G_{\mathbb{C}},G_{\mathbb{C}}')$, 
 parabolic subgroups $P_{\mathbb{C}}$ 
 satisfying the sphericity condition (iv) were classified in \cite{xhnoo}.  
See also \cite{xrims40, xtanaka12} 
 for some classification of strongly visible actions.  
\par\noindent(3)\enspace
Littelmann \cite{Li94} classified the pairs of parabolic subgroups $({P_1}_{\mathbb{C}}, {P_2}_{\mathbb{C}})$
 satisfying (iv) in Theorem \ref{thm:fmtensor} 
 under the assumption 
 that ${P_1}_{\mathbb{C}}$ and ${P_2}_{\mathbb{C}}$
 are maximal, 
 whereas all the pairs $({P_1}_{\mathbb{C}}, {P_2}_{\mathbb{C}})$
 satisfying the strong visibility condition (iii) in Theorem \ref{thm:fmtensor} were classified 
 in \cite{xrims40} for type A 
 and in Tanaka \cite{xtanaka12}
 for the other cases.  
\end{remark}

By Theorem \ref{thm:introQsph2}, 
 we are interested in the following question
 in connection with Theorem \ref{thm:mbrest}.  
\begin{question}
\label{q:para}
For which simple Lie group $G$, 
 does there exist a parabolic subgroup $P$
 such that $G_{\mathbb{C}}/P_{\mathbb{C}}$ is 
$G_U'$-strongly visible 
 (or equivalently, 
$G_{\mathbb{C}}'$-spherical)
 for {\it{all}} symmetric pairs $(G,G')$?
\end{question}

We give an affirmative answer 
 to this question
 if $G$ is the automorphism group
 of a para-Hermitian symmetric space.

Let $P=L N$ be a Levi decomposition
 of a parabolic subgroup $P$.  
Without loss of generality, 
 we may and do assume
 that both $G'$ and $L$ are stable 
 under the Cartan involution $\theta$ of $G$.  
We write $G_U$, $G_U'$, and $L_U$ 
 for the connected subgroups of $G_{\mathbb{C}}$
 with Lie algebra 
$
   {\mathfrak{g}}_U={\mathfrak{k}}+\sqrt{-1}{\mathfrak{p}}
$, 
$
   {\mathfrak{g}}_U':={\mathfrak{g}}_{\mathbb{C}}' \cap {\mathfrak{g}}_U
$, 
 and 
$
   {\mathfrak{l}}_U:={\mathfrak{l}}_{\mathbb{C}}\cap {\mathfrak{g}}_U
$.  
We note that $L_U = P_{\mathbb{C}} \cap G_U$.  
If the unipotent radical $N$ is abelian, 
 or equivalently, 
 if $(G,L)$ is a para-Hermitian symmetric pair, 
 then $G_U/L_U$ is a compact Hermitian symmetric space
 and the strong visibility condition (ii) in Theorem \ref{thm:introQsph2}
 for $G_U/L_U \simeq G_{\mathbb{C}}/P_{\mathbb{C}}$
 is satisfied
 for all symmetric pairs $(G_U, G_U')$
 by Fact \ref{fact:GKvisible}.  
Similarly
 the strong visibility condition (ii) in Theorem \ref{thm:fmtensor}
 for the tensor product case is satisfied 
 if the unipotent radicals of $P_1$ and $P_2$ are abelian
 \cite[Thm.\ 1.7]{visibleGK}.

Thus we have proved the following 
 in answer to Question \ref{q:para}:

\begin{corollary}
\label{cor:para}
Let $G$ be a non-compact simple Lie group, 
 and $G/L$ a para-Hermitian symmetric space.  
Then Theorem \ref{thm:mbrest} holds
 for any symmetric pair $(G,G')$
 by taking $\Pi$
 to be $\operatorname{Ind}_P^G(\xi)$ for $\xi \in \operatorname{Irr}(P)_f$.  
Likewise, 
 Theorem \ref{thm:tensor} holds 
 by taking $\Pi_1$ and $\Pi_2$ 
 to be $\operatorname{Ind}_P^G(\xi_1)$
 and $\operatorname{Ind}_P^G(\xi_2)$
 for $\xi_1, \xi_2 \in \operatorname{Irr}(P)_f$.  
\end{corollary}

\subsection{Para-Hermitian symmetric spaces}

Kaneyuki--Kozai \cite{kaneyuki85} gave 
 a classification of para-Hermitian symmetric pairs $({\mathfrak{g}}, {\mathfrak{l}})$
 for simple Lie algebras ${\mathfrak{g}}$
 as in Table \ref{table:4.1} below.  

\begin{table}[H]
\begin{center}
\begin{tabular}{cc}
${\mathfrak{g}}$
&${\mathfrak{l}}$
\\
\hline
${\mathfrak{s l}}(p+q,{\mathbb{R}})$
&${\mathfrak{s l}}(p,{\mathbb{R}}) + {\mathfrak{s l}}(q,{\mathbb{R}})+ {\mathbb{R}}$
\\
${\mathfrak{s u}}^{\ast}(2p+2q)$
&${\mathfrak{s u}}^{\ast}(2p) + {\mathfrak{s u}}^{\ast}(2q)+{\mathbb{R}}$
\\
${\mathfrak{s l}}(p+q,{\mathbb{C}})$
&${\mathfrak{s l}}(p,{\mathbb{C}}) + {\mathfrak{s l}}(q,{\mathbb{C}})+ {\mathbb{C}}$
\\
${\mathfrak{s u}}(n,n)$
&${\mathfrak{s l}}(n,{\mathbb{C}}) + {\mathbb{R}}$
\\
${\mathfrak{s o}}(n,n)$
&${\mathfrak{s l}}(n,{\mathbb{R}}) + {\mathbb{R}}$
\\
${\mathfrak{s o}}^{\ast}(4n)$
&${\mathfrak{s u}}^{\ast}(2n) + {\mathbb{R}}$
\\
${\mathfrak{s o}}(2n,{\mathbb{C}})$
&${\mathfrak{s l}}(n,{\mathbb{C}}) +  {\mathbb{C}}$
\\
${\mathfrak{s o}}(p+1,q+1)$
&${\mathfrak{s o}}(p,q) + {\mathbb{R}}$
\\
${\mathfrak{s o}}(n+2,{\mathbb{C}})$
&${\mathfrak{s o}}(n,{\mathbb{C}}) + {\mathbb{C}}$
\\
${\mathfrak{s p}}(n,{\mathbb{R}})$
&${\mathfrak{s l}}(n,{\mathbb{R}}) + {\mathbb{R}}$
\\
${\mathfrak{s p}}(n,n)$
&${\mathfrak{s u}}^{\ast}(2n) + {\mathbb{R}}$
\\
${\mathfrak{s p}}(n,{\mathbb{C}})$
&${\mathfrak{s l}}(n,{\mathbb{C}}) + {\mathbb{C}}$
\\
${\mathfrak{e}}_{6(6)}$
&${\mathfrak{s o}}(5,5) + {\mathbb{R}}$
\\
${\mathfrak{e}}_{6(-26)}$
&${\mathfrak{s o}}(1,9) +  {\mathbb{R}}$
\\
${\mathfrak{e}}_{6,{\mathbb{C}}}$
&${\mathfrak{s o}}(10, {\mathbb{C}}) + {\mathbb{C}}$
\\
 ${\mathfrak{e}}_{7(7)}$
& ${\mathfrak{e}}_{6(6)}+{\mathbb{R}}$
\\
${\mathfrak{e}}_{7(-25)}$
& ${\mathfrak{e}}_{6(-26)}+{\mathbb{R}}$
\\ 
${\mathfrak{e}}_{7,{\mathbb{C}}}$
& ${\mathfrak{e}}_{6,{\mathbb{C}}}+{\mathbb{C}}$
\\
\end{tabular}
\end{center}
\caption{List of para-Hermitian symmetric pairs
 with ${\mathfrak{g}}$ simple}
\label{table:4.1}
\end{table}

In particular, 
 Theorem \ref{thm:mbrest} holds 
 if ${\mathfrak{g}}$ is in Table \ref{table:4.1}.  
Similarly, 
 Theorem \ref{thm:tensor}
 for the tensor product representations hold
 if ${\mathfrak{g}}$ is in Table \ref{table:4.1}
(see \cite[Cor.\ 4.11]{K22}).

\section{Restriction of \lq\lq{smallest}\rq\rq\ representations}
\label{sec:min}

This section provides
 a bounded multiplicity theorem
 for the restriction $\Pi|_{G'}$ 
 when the associated variety ${\mathcal{V}}(\operatorname{Ann} \Pi)$
 of $\Pi \in \operatorname{Irr}(G)$ is the closure
 of the complex minimal nilpotent orbits.  
The main result of this section
 is Theorem \ref{thm:Joseph}
 which was proved in \cite{tkVarnaMin}
 under the assumption
 that ${\mathfrak{g}}$ is absolutely simple.  
We shall see that the same line of argument works
 when ${\mathfrak{g}}$ is a complex simple Lie algebra.  
At the end of this section, 
 we give a proof of Theorems \ref{thm:mbrest} and \ref{thm:tensor}
 for simple Lie algebras ${\mathfrak{g}}$ 
except for ${\mathfrak{s p}}(p,q)$ and ${\mathfrak{f}}_{4(-20)}$.

\subsection{Complex minimal nilpotent orbits ${\mathbb{O}}_{\operatorname{min}, {\mathbb{C}}}$}

Let ${\mathfrak{g}}_{\mathbb{C}}$ be a complex simple Lie algebra.  
There exists a unique non-zero minimal nilpotent
 $(\operatorname{Int}{\mathfrak{g}}_{\mathbb{C}})$-orbit
 in ${\mathfrak{g}}_{\mathbb{C}}^{\ast}$, 
 which we denote by ${\mathbb{O}}_{\operatorname{min}, {\mathbb{C}}}$.  
We write
 $n({\mathfrak{g}}_{\mathbb{C}})$
 for half the (complex) dimension of ${\mathbb{O}}_{\operatorname{min}, {\mathbb{C}}}$. 
Here is the formula of $n({\mathfrak{g}}_{\mathbb{C}})$, 
 see \cite{C93} for example.

\begin{figure}[H]
\begin{center}
\begin{tabular}{c|ccccccccc}
${\mathfrak{g}}_{\mathbb{C}}$
& $A_n$
& $B_n$ $(n \ge 2)$
& $C_n$
& $D_n$
& ${\mathfrak{g}}_{2}^{\mathbb{C}}$
& ${\mathfrak{f}}_{4}^{\mathbb{C}}$
& ${\mathfrak{e}}_{6}^{\mathbb{C}}$
& ${\mathfrak{e}}_{7}^{\mathbb{C}}$
& ${\mathfrak{e}}_{8}^{\mathbb{C}}$
\\
\hline
$n({\mathfrak{g}}_{\mathbb{C}})$
& $n$
& $2n-2$
& $n$
& $2n-3$
& $3$
& $8$
& $11$
& $17$
& $29$
\\
\end{tabular}
\end{center}
\end{figure}

Let $G$ be a non-compact connected simple Lie group
with Lie algebra ${\mathfrak{g}}$.  
We set ${\mathfrak{g}}_{\mathbb{C}}:={\mathfrak{g}}\otimes_{\mathbb{R}} {\mathbb{C}}$.  
We note that ${\mathfrak{g}}_{\mathbb{C}}$ is simple
 if ${\mathfrak{g}}$ does not have a complex structure.  
For a complex simple Lie algebra ${\mathfrak{g}}$, 
 we set $n({\mathfrak{g}}_{\mathbb{C}}):= 2 n({\mathfrak{g}})$.  
To see its meaning, 
 we write $J$ for the complex structure on ${\mathfrak{g}}$, 
 and decompose ${\mathfrak{g}}_{\mathbb{C}}={\mathfrak{g}} \otimes_{\mathbb{R}}{\mathbb{C}}$
 into the direct sum of the eigenspaces
 ${\mathfrak{g}}^{\operatorname{hol}}$ and ${\mathfrak{g}}^{\operatorname{anti}}$
 of $J$
 with eigenvalues
 $\sqrt{-1}$ and $-\sqrt{-1}$, 
respectively.  
Then one has a direct sum decomposition:
\[
 {\mathfrak{g}} \oplus {\mathfrak{g}} \overset \sim \to
 {\mathfrak{g}}^{\operatorname{hol}} \oplus {\mathfrak{g}}^{\operatorname{anti}} ={\mathfrak{g}}_{\mathbb{C}}, 
\quad
  (X,Y) \mapsto \frac 1 2 (X-\sqrt{-1} J X, Y+\sqrt{-1}JY).  
\]
Accordingly, 
 the complexification $G_{\mathbb{C}}$
 of the complex Lie group $G$
 is given by the totally real embedding 
\begin{equation}
\label{eqn:cpxcpx} 
\operatorname{diag} \colon G \hookrightarrow G \times G=:G_{\mathbb{C}}, 
\end{equation}
where the second factor is equipped 
 with the reverse complex structure.  
In this case, 
 we set 
$
   {\mathbb{O}}_{\operatorname{min}, {\mathbb{C}}}:= 
{\mathbb{O}}_{\operatorname{min}} \times {\mathbb{O}}_{\operatorname{min}}
$
 where ${\mathbb{O}}_{\operatorname{min}}$ is the minimal nilpotent orbit
 for ${\mathfrak{g}}$.

\subsection{Real minimal nilpotent orbits}
\label{subsec:realmin}

Let $G$ be a connected non-compact simple Lie group.  
Denote by ${\mathcal{N}}$ the nilpotent cone
 in ${\mathfrak{g}}$, 
 and ${\mathcal{N}}/G$ the set of nilpotent orbits, 
 which may be identified with nilpotent coadjoint orbits 
 in ${\mathfrak{g}}^{\ast}$ via the Killing form.  
The finite set ${\mathcal{N}}/G$ is a poset 
 with respect to the closure ordering, 
 and there are at most two minimal elements
 in $({\mathcal{N}} \setminus \{0\})/G$, 
 which we refer to as {\it{real minimal nilpotent (coadjoint) orbits}}.  
See \cite{B98, C93, KO15, O15} and references therein.  
The relationship with the complex minimal nilpotent orbits
 ${\mathbb{O}}_{\operatorname{min}, {\mathbb{C}}}$
 in 
$
   {\mathfrak{g}}_{\mathbb{C}}
$
 is summarized as below.

\begin{fact}
[see {\it{e.g.}}, {\cite{O15}}]
\label{lem:CRmin}
Let ${\mathfrak{g}}={\mathfrak{k}}+{\mathfrak{p}}$
 be a Cartan decomposition of a simple Lie algebra ${\mathfrak{g}}$.  
Then exactly one of the following cases occurs.  
\newline
{\rm{(1)}}\enspace
$({\mathfrak{g}}, {\mathfrak{k}})$ is not of Hermitian type, 
 and ${\mathbb{O}}_{\operatorname{min}, {\mathbb{C}}} \cap {\mathfrak{g}}
=\emptyset$.  
\newline
{\rm{(2)}}\enspace
$({\mathfrak{g}}, {\mathfrak{k}})$ is not of Hermitian type, 
 and ${\mathbb{O}}_{\operatorname{min}, {\mathbb{C}}} \cap {\mathfrak{g}}$
 is a single orbit of $G$.  
\newline
{\rm{(3)}}\enspace
$({\mathfrak{g}}, {\mathfrak{k}})$ is of Hermitian type, 
 and ${\mathbb{O}}_{\operatorname{min}, {\mathbb{C}}} \cap {\mathfrak{g}}$
 consists of two orbits of $G$.  
\end{fact}

Correspondingly, 
 we write
$
   {\mathbb{O}}_{\operatorname{min}, {\mathbb{C}}} \cap {\mathfrak{g}}
= \{ {\mathbb{O}}_{\operatorname{min}, {\mathbb{R}}}\}
$ in Case (2)
 of Fact \ref{lem:CRmin}, 
$
    {\mathbb{O}}_{\operatorname{min}, {\mathbb{C}}} \cap {\mathfrak{g}}
   = \{ {\mathbb{O}}_{\operatorname{min}, {\mathbb{R}}}^+, {\mathbb{O}}_{\operatorname{min}, {\mathbb{R}}}^-\}$ in Case (3).  
Then they exhaust
 all real minimal nilpotent orbits
 in Cases (2) and (3).  
Real minimal nilpotent orbits are unique in Case (1), 
 to be denoted by ${\mathbb{O}}_{\operatorname{min}, {\mathbb{R}}}$.  
We set
\begin{equation}
\label{eqn:ng}
 m({\mathfrak{g}}):=
\begin{cases}
\frac 1 2 \dim {\mathbb{O}}_{\operatorname{min}, {\mathbb{R}}}
\quad
&\text{in Cases (1) and (2)}, 
\\
\frac 1 2 \dim {\mathbb{O}}_{\operatorname{min}, {\mathbb{R}}}^+
=\frac 1 2 \dim {\mathbb{O}}_{\operatorname{min}, {\mathbb{R}}}^-
\quad
&\text{in Case (3)}.  
\end{cases}
\end{equation}

\begin{equation*}
{\mathbb{O}}_{\operatorname{min}, {\mathbb{R}}}^{\mathbb{C}}
:=
\begin{cases}
\operatorname{Ad}(G_{\mathbb{C}}) {\mathbb{O}}_{\operatorname{min}, {\mathbb{R}}}
&\text{in Cases (1) and (2), }
\\
\operatorname{Ad}(G_{\mathbb{C}})
{\mathbb{O}}_{\operatorname{min}, {\mathbb{R}}}^+
=
\operatorname{Ad}(G_{\mathbb{C}})
{\mathbb{O}}_{\operatorname{min}, {\mathbb{R}}}^-
\quad
&\text{in Case (3).  }
\end{cases}
\end{equation*}
Then $m({\mathfrak{g}})=n({\mathfrak{g}}_{\mathbb{C}})$
 in Cases (2) and (3), 
 and $m({\mathfrak{g}})>n({\mathfrak{g}}_{\mathbb{C}})$ in Case (1).  
The formula of $m({\mathfrak{g}})$ in Case (1) is given in 
 \cite{O15} as follows.  

\begin{figure}[H]
\begin{center}
\begin{tabular}{c|ccccc}
${\mathfrak{g}}$
& ${\mathfrak{s u}}^{\ast}(2n)$
& ${\mathfrak{s o}}(n-1,1)$
& ${\mathfrak{s p}}(m,n)$
& ${\mathfrak{f}}_{4(-20)}$
& ${\mathfrak{e}}_{6(-26)}$
\\
\hline
$m({\mathfrak{g}})$
& $4n-4$
& $n-2$
& $2(m+n)-1$
& $11$
& $16$
\\
\end{tabular}
\end{center}
\end{figure}

Here is a summary
 about when $m({\mathfrak{g}})>n({\mathfrak{g}}_{\mathbb{C}})$.

\begin{fact}
[{\cite{B98}, \cite[Cor.\ 5.9]{KO15}, \cite[Prop.\ 4.1]{O15}}]
\label{fact:OKO}
Suppose that ${\mathfrak{g}}$ is absolutely simple.  
Then the following six conditions
 on ${\mathfrak{g}}$ 
 are equivalent:
\begin{enumerate}
\item[{\rm{(i)}}]
${\mathbb{O}}_{\operatorname{min}} \cap {\mathfrak{g}}
 = \emptyset$.  
\item[{\rm{(ii)}}]
${\mathbb{O}}_{\operatorname{min}, {\mathbb{C}}} \ne 
  {\mathbb{O}}_{\operatorname{min}, {\mathbb{R}}}^{\mathbb{C}}$.  
\item[{\rm{(iii)}}]
$\theta \beta \ne -\beta$.  
\item[{\rm{(iv)}}]
$m({\mathfrak{g}})>n({\mathfrak{g}}_{\mathbb{C}})$.  
\item[{\rm{(v)}}]
${\mathfrak{g}}$ is compact
 or is isomorphic to ${\mathfrak{su}}^{\ast}(2n)$, 
 ${\mathfrak{so}}(n-1,1)$ $(n \ge 5)$, 
 ${\mathfrak{sp}}(m,n)$, 
 ${\mathfrak{f}}_{4(-20)}$, 
 or ${\mathfrak{e}}_{6(-26)}$.  
\item[{\rm{(vi)}}]
${\mathfrak{g}}_{\mathbb{C}}={\mathfrak{k}}_{\mathbb{C}}$
 or the pair $({\mathfrak{g}}_{\mathbb{C}}, {\mathfrak{k}}_{\mathbb{C}})$
 is isomorphic to 
$({\mathfrak{s l}}(2n, {\mathbb{C}}), {\mathfrak{s p}}(n,{\mathbb{C}}))$, 
$({\mathfrak{s o}}(n, {\mathbb{C}}), {\mathfrak{s o}}(n-1,{\mathbb{C}}))$ 
 $(n \ge 5)$, 
$({\mathfrak{s p}}(m+n, {\mathbb{C}}), {\mathfrak{s p}}(m,{\mathbb{C}}) \oplus {\mathfrak{s p}}(n, {\mathbb{C}}))$, 
$({\mathfrak{f}}_4^{\mathbb{C}}, {\mathfrak{s o}}(9,{\mathbb{C}}))$, 
 or 
$({\mathfrak{e}}_{6}^{\mathbb{C}}, {\mathfrak{f}}_{4}^{\mathbb{C}})$.  
\end{enumerate}
\end{fact}

\begin{remark}
The equivalence (i) $\iff$ (v) was stated 
 in \cite[Prop.\ 4.1]{B98}
 without proof.  
One may find a proof in \cite{O15}.  
\end{remark}

\subsection{Gelfand--Kirillov dimension}
\label{sec:GKdim}

The Gelfand--Kirillov dimension serves
 as a coarse measure of the \lq\lq{size}\rq\rq\
 of representations.  
Let $G$ be a real reductive Lie group.  
We recall from Section \ref{sec:2}
 that for $\Pi \in {\mathcal{M}}(G)$, 
we denote by $\operatorname{Ann}\Pi$ the annihilator of $\Pi$
 in the universal enveloping algebra $U({\mathfrak{g}}_{\mathbb{C}})$
 of the complexified Lie algebra ${\mathfrak{g}}_{\mathbb{C}}$.  
The associated variety ${\mathcal{V}}(\operatorname{Ann}\Pi)$ is 
 the closure of a single nilpotent coadjoint orbit
 in ${\mathfrak{g}}_{\mathbb{C}}^{\ast}$
 if $\Pi \in \operatorname{Irr}(G)$.  
The Gelfand--Kirillov dimension $\operatorname{DIM}(\Pi)$ of $\Pi$ 
 is defined to be half the dimension
 of ${\mathcal{V}}(\operatorname{Ann}\Pi)$.  
The same notation will be applied for Harish-Chandra modules
 of finite length.

By definition, 
 the Gelfand--Kirillov dimension has the following property:
\[
\operatorname{DIM}(\Pi)=0
\,\,\iff\,\,
\Pi
\text{ is finite-dimensional.  }
\]

For any infinite-dimensional $\Pi \in \operatorname{Irr}(G)$, 
 one has 
\begin{equation}
\label{eqn:GKbdd}
(n({\mathfrak{g}}_{\mathbb{C}})\le)
\,\,
m({\mathfrak{g}}) \le \operatorname{DIM}(\Pi).  
\end{equation}

\subsection{Coisotropic action on ${\mathbb{O}}_{\operatorname{min}, {\mathbb{C}}}$}
As we saw in Section \ref{subsec:coiso}, 
 any coadjoint orbit of a Lie group $G$
 is a Hamiltonian $G$-manifold
 with the Kirillov--Kostant--Souriau symplectic form.  
We consider the holomorphic setting, 
 and have proved in \cite[Thm.\ 23] 
{tkVarnaMin} the following:

\begin{fact}
\label{fact:22030921}
Let ${\mathbb{O}}_{\operatorname{min},{\mathbb{C}}}$ be the minimal nilpotent coadjoint orbit of a connected complex simple Lie group $G_{\mathbb{C}}$.  
\begin{enumerate}
\item[{\rm{(1)}}]
For any symmetric pair $(G_{\mathbb{C}}, K_{\mathbb{C}})$, 
 the $K_{\mathbb{C}}$-action 
 on ${\mathbb{O}}_{\operatorname{min},{\mathbb{C}}}$
 is coisotropic.  
\item[{\rm{(2)}}]
The diagonal action of $G_{\mathbb{C}}$
 on 
$
   {\mathbb{O}}_{\operatorname{min},{\mathbb{C}}} \times {\mathbb{O}}_{\operatorname{min},{\mathbb{C}}}
$
 is coisotropic.  
\end{enumerate}
\end{fact}

In Section \ref{sec:7}, 
 we give a generalization of this statement, 
 see Theorems \ref{thm:22040624} and \ref{thm:22100105}.  

\subsection{Bounded multiplicity theorems}

In view of the inequality \eqref{eqn:GKbdd}, 
 one may think of $\Pi \in \operatorname{Irr}(G)$ satisfying
 $\operatorname{DIM}(\Pi)=n({\mathfrak{g}}_{\mathbb{C}})$
 as the \lq\lq{smallest}\rq\rq\
 amongst infinite-dimensional irreducible representations of $G$.  
Minimal representations \cite{GS05, J79, Ta} are unitarizable
 and have this property.  
For $G=SL(n,{\mathbb{R}})$, 
 $SL(n,{\mathbb{C}})$, or $SU(p,q)$ $(p,q>0)$, 
 the Joseph ideal is not defined, 
 but there exist infinitely many irreducible unitary representations $\Pi$
 with $\operatorname{DIM}(\Pi)=n({\mathfrak{g}}_{\mathbb{C}})$.  
In general, 
 the coherent continuation of such representations
 obtained by the tensor product 
 with finite-dimensional representations also satisfy $\operatorname{DIM}(\Pi) =n({\mathfrak{g}}_{\mathbb{C}})$.

The restriction of such $\Pi$ to arbitrary symmetric pairs $(G,G')$
 has  a bounded multiplicity property
 as follows.  
\begin{theorem}
\label{thm:Joseph}
Let $G$ be a connected simple Lie group, 
 and $\Pi$, $\Pi_1$, $\Pi_2 \in \operatorname{Irr}(G)$. 
\begin{enumerate}
\item[{\rm{(1)}}]
If $\operatorname{DIM}(\Pi)=n({\mathfrak{g}}_{\mathbb{C}})$, 
 then  for any symmetric pair $(G,G')$, 
 one has
\[
\underset{\pi \in \operatorname{Irr}(G')}{\sup}[\Pi|_{G'}:\pi]<\infty.  
\]
\item[{\rm{(2)}}]
If $\operatorname{DIM}(\Pi_1)=\operatorname{DIM}(\Pi_2)= n({\mathfrak{g}}_{\mathbb{C}})$, 
 then one has
\[
\underset{\Pi \in \operatorname{Irr}(G)}{\sup}
 [{\Pi_1 \otimes \Pi_2}:{\Pi}] <\infty.  
\] 
\end{enumerate}
\end{theorem}

\begin{remark}
\label{rem:5.6}
{\rm{(1)}}\enspace
When $(G,G')$ is a Riemannian symmetric pair, 
 namely, 
 $G'=K$, 
 Theorem \ref{thm:Joseph} (1)
 for minimal representations $\Pi$ is known 
 by Kostant
 in a stronger from that the supremum is one, 
 see \cite[Prop.\ 4.10]{GS05}.  
\par\noindent
{\rm{(2)}}\enspace
Theorem \ref{thm:Joseph} was proved 
in \cite[Thms.\ 7 and 8]  
{tkVarnaMin}
 by using Fact \ref{fact:22030921}
 when ${\mathfrak{g}}$ is absolutely simple.  
\end{remark}

\begin{remark}
We shall see in Theorem \ref{thm:mgbdd}
 and Remark \ref{rem:OK}
 that Theorem \ref{thm:Joseph} still holds
 by replacing $n({\mathfrak{g}}_{\mathbb{C}})$ with $m({\mathfrak{g}})$.  
\end{remark}

\begin{proof}[Proof of Theorem \ref{thm:Joseph}]
As we saw in Remark \ref{rem:5.6}, 
 it suffices to consider when $G$ is a complex Lie group.  
In this case
 there are two types of involutions $\sigma$ of $G$:
\newline\indent
(1)\enspace($\sigma$ is holomorphic)\quad\quad\,\,
$G^{\sigma}$ is a complex subgroup of $G$, 
\newline\indent
(2)\enspace($\sigma$ is anti-holomorphic)\enspace
$G^{\sigma}$ is a real form of $G$.

For simplicity, 
 suppose that $G'$ is the identity component of $G^{\sigma}$.  
Then via the identification 
 $G_{\mathbb{C}} \simeq G \times G$ in \eqref{eqn:cpxcpx}, 
 one has 
\begin{alignat*}{3}
G'_{\mathbb{C}} &\simeq && G' \times G'
&&\text{for (1)}, 
\\
G_{\mathbb{C}}' &\simeq && \operatorname{diag}_{\sigma}(G)
:=\{(g, \sigma g):g \in G\}
\quad
&&\text{for (2)}.  
\end{alignat*}

Then $G_{\mathbb{C}}'$ acts on ${\mathbb{O}}_{\operatorname{min}, {\mathbb{C}}} \times {\mathbb{O}}_{\operatorname{min}, {\mathbb{C}}}$
 coisotropically 
 in both cases (1) and (2)
 by Fact \ref{fact:22030921}
 (1) and (2), 
 respectively.  
This implies the first statement of Theorem \ref{thm:Joseph} by Theorem \ref{thm:Ki}.  
On the other hand, 
 $G_{\mathbb{C}} \times G_{\mathbb{C}}$ acts on 
 $({\mathbb{O}}_{\operatorname{min}, {\mathbb{C}}} \times {\mathbb{O}}_{\operatorname{min}, {\mathbb{C}}}) \times ({\mathbb{O}}_{\operatorname{min}, {\mathbb{C}}} \times {\mathbb{O}}_{\operatorname{min}, {\mathbb{C}}})$
 coisotropically
 by Fact \ref{fact:22030921} (2), 
 whence the second statement of Theorem \ref{thm:Joseph} follows.  
\end{proof}

\subsection{Proof of Theorems \ref{thm:mbrest}--\ref{thm:tensor}
 except ${\mathfrak{s p}}(p,q)$ and ${\mathfrak{f}}_{4(-20)}$}

In order to apply Theorem \ref{thm:Joseph}, 
 we need the existence of $\Pi \in \operatorname{Irr}(G)$
 satisfying $\operatorname{DIM}(\Pi)=n({\mathfrak{g}}_{\mathbb{C}})$.  
However, 
 we know from the inequality \eqref{eqn:GKbdd}
 that there is no such $\Pi$
 if $m({\mathfrak{g}}) > n({\mathfrak{g}}_{\mathbb{C}})$, 
 namely, 
 if ${\mathfrak{g}}$ is in the list of Fact \ref{fact:OKO} (v).  
The converse is not true, 
 but \lq\lq{almost}\rq\rq\ holds as follows.

\begin{lemma}
\label{lem:minrep}
Let $G$ be a simply-connected non-compact simple Lie group.  
Then there exist an infinite-dimensional irreducible and unitarizable representation $\Pi$ of $G$
 such that 
$
   \operatorname{DIM}(\Pi) = n({\mathfrak{g}}_{\mathbb{C}})
$, 
 if ${\mathfrak{g}}$ is not isomorphic to the following:

\begin{align*}
&{\mathfrak{s o}}(n,1)\,\, (n \ge 6),\quad
{\mathfrak{s o}}(p,q)\,\,(p,q \ge 4, p+q \,\text{odd}),\quad
\\
&
{\mathfrak{s u}}^{\ast}(2n), \quad
{\mathfrak{s p}}(p,q)\,\,(p,q \ge 1), \quad
{\mathfrak{e}}_{6(-26)}, \quad
{\mathfrak{f}}_{4(-20)}.  
\end{align*}
\end{lemma}

\begin{proof}
When ${\mathfrak{g}}$ is of type A, 
 one may take $\Pi$
 to be a  degenerate principal series representation 
 induced from a mirabolic subgroup 
 for ${\mathfrak{g}}={\mathfrak{s l}}(n, {\mathbb{F}})$
 (${\mathbb{F}}={\mathbb{R}}, {\mathbb{C}}$), 
 and a highest weight module
 of the smallest Gelfand--Kirillov dimension
 for ${\mathfrak{g}}={\mathfrak{s u}}(p,q)$.  
When ${\mathfrak{g}}$ is not of type A, 
 one may take $\Pi$ to be a minimal representation \cite{Ta}.  
\end{proof}

\begin{proof}[Proof of Theorems \ref{thm:mbrest} and \ref{thm:tensor} except for 
${\mathfrak{g}}= {\mathfrak{s p}}(p,q)$ and ${\mathfrak{f}}_{4(-20)}$]
If ${\mathfrak{g}}$ is not in the list of Lemma \ref{lem:minrep}, 
 there exists $\Pi \in \operatorname{Irr}(G)$
 such that 
$\operatorname{DIM}(\Pi) = n ({\mathfrak{g}}_{\mathbb{C}})$.  
Hence Theorem \ref{thm:Joseph} applies.  
For ${\mathfrak{g}}={\mathfrak{s u}}^{\ast}(2n)$, ${\mathfrak{s o}}(p,q)$
 or ${\mathfrak{e}}_{6(-26)}$, 
 one sees from Table \ref{table:4.1}
 that $G$ is the transformation group 
 of a para-Hermitian symmetric space, 
hence Corollary \ref{cor:para} applies.  
\end{proof}

\section{Restriction of \lq\lq{small}\rq\rq\ representations}
\label{sec:7}

This section completes the proof of Theorems \ref{thm:mbrest}
 and \ref{thm:tensor}.  
As we have seen, 
 the remaining cases are 
 when ${\mathfrak{g}} = {\mathfrak{s p}}(p,q)$
 and ${\mathfrak{f}}_{4(-20)}$, 
 for which there is no $\Pi \in \operatorname{Irr}(G)$
with $\operatorname{DIM}(\Pi)=n({\mathfrak{g}}_{\mathbb{C}})$
 and for which $G$ does not admit a Hermitian or para-Hermitian symmetric space, 
hence none of Theorem \ref{thm:mfB}, 
 Corollary \ref{cor:para}, 
 or Theorem \ref{thm:Joseph} applies.  
By the classification of irreducible symmetric pairs (Berger \cite{Be57}), 
 we need to treat the following symmetric pairs
 $({\mathfrak{g}}, {\mathfrak{g}}')$:

\begin{table}[H]
\begin{center}
\begin{tabular}{c|c}
${\mathfrak{g}}$
&${\mathfrak{g}}'$
\\
\hline
${\mathfrak{s p}}(p,q)$
&${\mathfrak{u}}(p,q)$, ${\mathfrak{s p}}(p_1,q_1) + {\mathfrak{s p}}(p-p_1,q-q_1)$
\\
${\mathfrak{f}}_{4(-20)}$
&${\mathfrak{s o}}(9)$, ${\mathfrak{s o}}(8,1)$, ${\mathfrak{s p}}(2,1)+{\mathfrak{s p}}(1)$
\\
\end{tabular}
\end{center}
\caption{Remaining symmetric pairs}
\end{table}

The main results of this section is Theorem \ref{thm:mgbdd}, 
 which guarantees 
 the bounded multiplicity property
 for the restriction $\Pi|_{G'}$
 for any $\Pi \in \operatorname{Irr}(G)$ satisfies
 $\operatorname{DIM}(\Pi)=m({\mathfrak{g}})\,\,(> n({\mathfrak{g}}_{\mathbb{C}}
))$, 
 and we complete the proof of Theorems \ref{thm:mbrest} and \ref{thm:tensor}
 in the end.

\subsection{Bounded multiplicity theorems}

Suppose that $(G,G')$ is a symmetric pair
 defined by an involution $\sigma$ of $G$.  
We use the same letter $\sigma$ 
 to denote its holomorphic extension
 to a simply connected complexification $G_{\mathbb{C}}$, 
 and also its differential.  
We set ${\mathfrak{g}}^{-\sigma}:=\{Y \in {\mathfrak{g}}: \sigma Y=-Y \}$.  
We take a Cartan involution $\theta$
 commuting with $\sigma$, 
 and write ${\mathfrak{g}}={\mathfrak{k}}+{\mathfrak{p}}$
 for the Cartan decomposition.  
We take a maximal split abelian subspace ${\mathfrak{a}}$ to be $\sigma$-split, 
 namely, 
 ${\mathfrak{a}}^{-\sigma}:={\mathfrak{a}} \cap {\mathfrak{g}}^{-\sigma}$ is
 a maximal abelian subspace
 in ${\mathfrak{p}} \cap {\mathfrak{g}}^{-\sigma}$, 
 and $\Sigma^+({\mathfrak{g}}, {\mathfrak{a}})$
 to be compatible 
 with a positive system $\Sigma^+({\mathfrak{g}}, {\mathfrak{a}}^{-\sigma})$.  
Let $\mu$ be the highest element in $\Sigma^+({\mathfrak{g}}, {\mathfrak{a}})$. We prove:

\begin{theorem}
\label{thm:mgbdd}
Suppose $\Pi\in \operatorname{Irr}(G)$ satisfies
 $\operatorname{DIM}(\Pi)=m({\mathfrak{g}})$.  
If $\sigma \mu=-\mu$, 
 then the restriction $\Pi|_{G'}$ has 
 the bounded multiplicity property \eqref{eqn:bddrest}.  
\end{theorem}

This theorem extends 
 \cite[Thm.\ 34] 
{tkVarnaMin}, 
 which treated $\sigma=\theta$ (Cartan involution)
 or its conjugation by $\operatorname{Int}({\mathfrak{g}}_{\mathbb{C}})$.

\begin{example}
\label{ex:6.2}
(1)\enspace
The assumption $\sigma \mu=-\mu$ in Theorem \ref{thm:mgbdd}
 is automatically satisfied
 if ${\mathfrak{a}}^{-\sigma}={\mathfrak{a}}$, 
 namely, 
 if $\operatorname{rank}_{\mathbb{R}} G/G'=\operatorname{rank}_{\mathbb{R}} G$.  
This is the case $({\mathfrak{g}}, {\mathfrak{g}}')=({\mathfrak{s p}}(p,q), {\mathfrak{u}}(p,q))$
 or ${\mathfrak{g}}={\mathfrak{f}}_{4(-20)}$.  
\newline
(2)\enspace
A direct computation shows $\sigma \mu=-\mu$
 for 
 $({\mathfrak{g}}, {\mathfrak{g}}')=({\mathfrak{s p}}(p,q), {\mathfrak{s p}}(p_1,q_1) + {\mathfrak{s p}}(p-p_1, q-q_1))$.  
\end{example}

\begin{remark}
\label{rem:OK}
(1)\enspace
If $m({\mathfrak{g}})=n({\mathfrak{g}}_{\mathbb{C}})$, 
  the conclusion of Theorem \ref{thm:mgbdd} holds 
 without the assumption $\sigma \mu=-\mu$, 
 see Theorem \ref{thm:Joseph}.  
\newline
(2)\enspace
Okuda \cite{OK22} verified 
 that the assumption $\sigma \mu=-\mu$ is satisfied
 for all symmetric pairs $({\mathfrak{g}}, {\mathfrak{g}}')$
 if ${\mathfrak{g}}$ is one of the five simple Lie algebras 
 in Fact \ref{fact:OKO} (v), 
 namely, 
 if $m({\mathfrak{g}}) > n({\mathfrak{g}}_{\mathbb{C}})$.  
\end{remark}

\begin{remark}
When $m({\mathfrak{g}})=n({\mathfrak{g}}_{\mathbb{C}})$, 
 it may happen 
 that $\sigma \mu \ne -\mu$.  
Here are examples of such symmetric pairs.  

\par\noindent
(1)\enspace
$({\mathfrak{s l}}(2n, {\mathbb{R}}), {\mathfrak{s p}}(n,{\mathbb{R}}))$

\par\noindent
(2)\enspace
$({\mathfrak{s u}}(2p, 2q), {\mathfrak{s p}}(p,q))$, 
$({\mathfrak{s u}}(n,n),{\mathfrak{s p}}(n,{\mathbb{R}}))$
\par\noindent
(3)\enspace
$({\mathfrak{s p}}(p+q,{\mathbb{R}}), {\mathfrak{s p}}(p,{\mathbb{R}})\oplus{\mathfrak{s p}}(q,{\mathbb{R}}))$, 
$({\mathfrak{s p}}(2n,{\mathbb{R}}), {\mathfrak{s p}}(n,{\mathbb{C}}))$, 

\par\noindent
(4)\enspace
$({\mathfrak{s o}}(p,q), {\mathfrak{s o}}(p-1,q))$
 or 
$({\mathfrak{s o}}(p,q),{\mathfrak{s o}}(p,q-1))$
 for \lq\lq{$p \ge q \ge 4$ and $p\equiv q \mod 2$}\rq\rq, 
 \lq\lq{$p \ge 5$ and $q=2$}\rq\rq,  
 or 
 \lq\lq{$p \ge 4$ and $q=3$}\rq\rq.  
\par\noindent
(5)\enspace
$({\mathfrak{f}}_{4(4)}, {\mathfrak{s o}}(5,4))$, 
\par\noindent
(6)\enspace
$({\mathfrak{e}}_{6(6)}, {\mathfrak{f}}_{4(4)})$, 
 $({\mathfrak{e}}_{6(2)}, {\mathfrak{f}}_{4(4)})$, 
 or 
$({\mathfrak{e}}_{6(-14)},{\mathfrak{f}}_{4(-20)})$, 
\par\noindent
(7)\enspace
complex symmetric pairs
 in Fact \ref{fact:OKO} (vi).

The condition $\sigma \mu \ne -\mu$ yields
 an interesting phenomenon 
 that the restriction
$\Pi|_{G'}$ stays almost irreducible
 (\cite[Thm.\ 10] 
{tkVarnaMin})
 for any $\Pi \in \operatorname{Irr}(G)$
 such that $\operatorname{DIM}(\Pi)=n({\mathfrak{g}}_{\mathbb{C}})$.  
 (Such $\Pi$ exists in the above cases (1)--(7).)  
This gives a uniform explanation
 of the phenomena
 that have been observed in various literatures, 
 for instance, 
 as a well-known property of the Segal--Shale--Weil representation
 of the metaplectic group
 for the pair (3), 
 the branching law of the minimal representation
 of $O(p,q)$
 in \cite[Thm.\ A]{KOr} for (4), 
 that of a degenerate principal series representation from a mirabolic 
 in \cite[Thm.\ 7.3]{KOPU09}
 for (1)
 (see also \cite{Clare12} for the complex case), 
 that of a cohomological parabolic induction $A_{\mathfrak{q}}(\lambda)$
 in \cite[Thm.\ 3.5]{K11Zuckerman} for (2), 
 and that of a minimal highest weight module
 in Binegar--Zierau \cite{BZ} 
 for $({\mathfrak{e}}_{6(-14)}, {\mathfrak{f}}_{4(-20)})$
 of (6), etc.  
\end{remark}

\subsection{Structural results on real minimal nilpotent orbits}

Retain the notation as in Section \ref{subsec:realmin}.  
The assumption $\operatorname{DIM}(\Pi)=m({\mathfrak{g}})$ means
 that the associated variety ${\mathcal{V}}(\operatorname{Ann}\Pi)$ of $\Pi$
 is the closure of ${\mathbb{O}}_{\operatorname{min}, {\mathbb{R}}}^{\mathbb{C}}$.  
In this section, 
 we recall some basic facts 
 on real minimal nilpotent orbits.

Let ${\mathfrak{g}}={\mathfrak{k}}+{\mathfrak{p}}$
 be a Cartan decomposition
 of a simple Lie algebra ${\mathfrak{g}}$.  
We take a maximal abelian subspace ${\mathfrak{a}}$
 of ${\mathfrak{p}}$, 
 and fix a positive system $\Sigma^+({\mathfrak{g}}, {\mathfrak{a}})$
 of the restricted root system $\Sigma({\mathfrak{g}}, {\mathfrak{a}})$.  
Let ${\mathfrak{m}}$ be
 the centralizer of ${\mathfrak{a}}$ in ${\mathfrak{k}}$.  
Denote by $\mu$
 the highest element 
 in $\Sigma^+({\mathfrak{g}}, {\mathfrak{a}})$, 
 and $A_{\mu} \in {\mathfrak{a}}$
 the coroot of $\mu$.  
Any (real) minimal nilpotent coadjoint orbit 
${\mathbb{O}}$
 is of the form 
 ${\mathbb{O}}=\operatorname{Ad}(G)X$
 for some non-zero element 
\[
   X \in {\mathfrak{g}}({\mathfrak{a}};\mu)
   :=
\{X \in {\mathfrak{g}}
:
[H,X]=\mu(H) X
\text{ for all $H \in {\mathfrak{a}}$}
\}
\]
via the identification 
 ${\mathfrak{g}}^{\ast} \simeq {\mathfrak{g}}$, 
 and vice versa
 ({\it{e.g.}}, \cite{O15}).  
Let $G_X$ be the stabilizer subgroup 
 of $X$ in $G$, 
 and ${\mathfrak{Z}}_{\mathfrak{g}}(X)$ its Lie algebra.  
We take $Y \in {\mathfrak{g}}({\mathfrak{a}};-\mu)$
 such that 
$\{X, A_{\mu}, Y\}$ forms
 an ${\mathfrak{s l}}_2$-triple.  
We write ${\mathfrak{s l}}_2^X$
 for the corresponding subalgebra
 in ${\mathfrak{g}}$.  
Since $\mu$ is the highest root
 in $\Sigma^+({\mathfrak{g}}, {\mathfrak{a}})$, 
 the representation theory of ${\mathfrak{s l}}_2({\mathbb{R}})$ tells us
 that possible eigenvalues 
 of $\operatorname{ad}(A_{\mu})$
 are $0$, $\pm 1$, or $\pm 2$. 
Hence one has the eigenspace decomposition 
 of $\operatorname{ad}(A_{\mu})$
 as 
\begin{equation}
\label{eqn:g012}
{\mathfrak{g}}
=
{\mathfrak{g}}_{-2}
\oplus
{\mathfrak{g}}_{-1}
\oplus
{\mathfrak{g}}_0
\oplus
{\mathfrak{g}}_1
\oplus
{\mathfrak{g}}_2, 
\end{equation}
where 
${\mathfrak{g}}_j:=\operatorname{Ker}
(\operatorname{ad}(A_{\mu})-j)$.  
We note 
 that ${\mathfrak{g}}_{\pm 2}={\mathfrak{g}}({\mathfrak{a}};\pm \mu)$.  
Let $a:=\dim_{\mathbb{R}}{\mathfrak{g}}_1$
 and $b:=\dim_{\mathbb{R}}{\mathfrak{g}}_2$.  
We denote by 
 ${\mathfrak{Z}}_{\mathfrak{g}}({\mathfrak{s l}}_2^X)$
 the centralizers of ${\mathfrak{s l}}_2^X$
 in ${\mathfrak{g}}$.

\begin{example}
\begin{enumerate}
\item[(1)]
For ${\mathfrak{g}}={\mathfrak{s p}}(p,q)$, 
 $a=4(p+q-2)$, $b=3$, 
$
{\mathfrak{Z}}_{\mathfrak{g}}({\mathfrak{s l}}_2^X)
 \simeq {\mathfrak{sp}}(p-1,q-1) \oplus {\mathbb{R}}$, 
 and ${\mathfrak{g}}_0 \simeq {\mathfrak{sp}}(p-1,q-1) \oplus {\mathfrak{sp}}(1)\oplus {\mathbb{R}}.
$

\item[(2)]
For ${\mathfrak{g}}={\mathfrak{f}}_{4(-20)}$, 
 one has $a=8$, $b=7$, 
 ${\mathfrak{Z}}_{\mathfrak{g}}({\mathfrak{s l}}_2^X)
 \simeq {\mathfrak{spin}}(6)$, 
 and ${\mathfrak{g}}_0 \simeq {\mathfrak{spin}}(7) \oplus {\mathbb{R}}$.  
\end{enumerate}
\end{example}

\begin{lemma}
\label{lem:22040704}
\begin{enumerate}
\item[{\rm{(1)}}]
The Lie algebra ${\mathfrak{g}}$ decomposes
as an ${\mathfrak{s l}}_2^{X}$-module:
\begin{equation}
\label{eqn:gsl2}
   {\mathfrak{g}}
  \simeq 
{\mathfrak{Z}}_{\mathfrak{g}}({\mathfrak{s l}}_2^X)
  \oplus a {\mathbb{R}}^2 \oplus b {\mathbb{R}}^3, 
\end{equation}
where ${\mathbb{R}}^2$ and ${\mathbb{R}}^3$
 stand for the natural representation
 and the adjoint representation of ${\mathfrak{s l}}_2({\mathbb{R}})$, 
 respectively.  
\item[{\rm{(2)}}]
One has a direct sum decomposition
 as a vector space:
\begin{equation}
\label{eqn:gX}
{\mathfrak{Z}}_{\mathfrak{g}}(X)
= {\mathfrak{Z}}_{\mathfrak{g}}
({\mathfrak{s l}}_2^X)
\oplus {\mathfrak{g}}_1
\oplus{\mathfrak{g}}_2.  
\end{equation}

\item[{\rm{(3)}}]
The dimension of the adjoint orbit
 $\operatorname{Ad}(G)X$ is equal to $a+2b$.  
\item[{\rm{(4)}}]
${\mathfrak{g}}_0
=
{\mathfrak{Z}}_{\mathfrak{g}}
({\mathfrak{s l}}_2^X)+({\mathfrak{m}} \oplus {\mathbb{R}}A_{\mu})$.  
\end{enumerate}
\end{lemma}

\begin{proof}
The first two assertions are immediate consequences
 of the representation theory of ${\mathfrak{s l}}_2({\mathbb{R}})$, 
 whence the dimension formula of ${\mathfrak{g}}/{\mathfrak{Z}}_{\mathfrak{g}}(X)$.  
For the last assertion, 
 the inclusion
${\mathfrak{g}}_0
\supset
{\mathfrak{Z}}_{\mathfrak{g}}
({\mathfrak{s l}}_2^X)+({\mathfrak{m}} \oplus {\mathbb{R}}A_{\mu})$
 is obvious.  
By the irreducible decomposition \eqref{eqn:gsl2}
 of ${\mathfrak{g}}$ as an ${\mathfrak{s l}}_2^X$-module, 
 one sees
 that 
${\mathfrak{g}}_0= {\mathfrak{Z}}_{\mathfrak{g}}({\mathfrak{s l}}_2^X)
\oplus
[X, {\mathfrak{g}}({\mathfrak{a}};-\mu)]$.  
Since $[{\mathfrak{g}}({\mathfrak{a}};\mu),
{\mathfrak{g}}({\mathfrak{a}};-\mu)]
\subset {\mathfrak{m}}+{\mathfrak{a}}$
 and since ${\mathfrak{a}} \subset {\mathbb{R}} A_{\mu} + {\mathfrak{Z}}_{\mathfrak{g}}({\mathfrak{s l}}_2^X)$, 
 the opposite inclusion follows.  
\end{proof}

\subsection{Coisotropic actions of $G_{\mathbb{C}}^{\sigma}$ on ${\mathbb{O}}_{\operatorname{min},{\mathbb{R}}}^{\mathbb{C}}$.} 

The proof of Theorem \ref{thm:mgbdd} is reduced
 to the following geometric properties
 by Theorem \ref{thm:Ki}.  

\begin{theorem}
\label{thm:22040624}
Assume $\sigma$ satisfies $\sigma \mu=-\mu$
 as in Theorem \ref{thm:mgbdd}.   
Then the $G_{\mathbb{C}}^{\sigma}$-action
 on ${\mathbb{O}}_{\operatorname{min}, {\mathbb{R}}}^{\mathbb{C}}$ is coisotropic.  
\end{theorem}

\begin{theorem}
\label{thm:22100105}
The diagonal action of $G_{\mathbb{C}}$
 on ${\mathbb{O}}_{\operatorname{min}, {\mathbb{R}}}^{\mathbb{C}} \times {\mathbb{O}}_{\operatorname{min}, {\mathbb{R}}}^{\mathbb{C}}$ is coisotropic.  
\end{theorem}

\begin{remark}
(1)\enspace
Theorem \ref{thm:22040624} generalizes \cite[Thm.~29]
{tkVarnaMin} 
 which treats the case $\sigma = \theta$
(Cartan involution).  
\newline
(2)\enspace
Theorem \ref{thm:22100105}
 generalizes Fact \ref{fact:22030921} (2)
 which treats the case 
 $m({\mathfrak{g}})=n({\mathfrak{g}}_{\mathbb{C}})$.  
\end{remark}

We take $X \in {\mathfrak{g}}({\mathfrak{a}};{\mu})$
 such that ${\mathbb{O}}_{\operatorname{min},{\mathbb{R}}}
 = \operatorname{Ad}(G)X$, 
 hence ${\mathbb{O}}_{\operatorname{min},{\mathbb{R}}}^{\mathbb{C}}
= \operatorname{Ad}(G_{\mathbb{C}})X$
 via the isomorphism ${\mathfrak{g}}_{\mathbb{C}}^{\ast} \simeq {\mathfrak{g}}_{\mathbb{C}}$.  

For the proof of Theorem \ref{thm:22040624}, 
 we begin with the following:

\begin{lemma}
\label{lem:22040623}
If $\sigma \mu=-\mu$
 then 
\[
   {\mathfrak{g}}
={\mathfrak{g}}^{\sigma}
  +({\mathfrak{m}}+{\mathbb{R}}A_{\mu})
 +{\mathfrak{Z}}_{\mathfrak{g}}(X)
={\mathfrak{g}}^{\sigma}+{\mathfrak{Z}}_{\mathfrak{g}}(X)+[X, {\mathfrak{g}}_{-2}].  
\]
\end{lemma}
\begin{proof}
Since $\sigma(A_{\mu})=-A_{\mu}$, 
 one has 
$\sigma({\mathfrak{g}}_j)={\mathfrak{g}}_{-j}$
 ($j=0,1,2$).  
We set $Y:= \sigma X \in {\mathfrak{g}}_{-2} = {\mathfrak{g}}({\mathfrak{a}};-\mu)$.  
Then $[X,Y] \in {\mathfrak{g}}({\mathfrak{a}}; 0) \cap {\mathfrak{g}}^{-\sigma}
={\mathfrak{a}}^{-\sigma}$.  
Hence, 
 after an appropriate normalization, 
 $\{X, \sigma X, [X, \sigma X]\}$ forms an ${\mathfrak{sl}}_2$-triple.  
In particular, 
 $[X, {\mathfrak{g}}_{-2}]=[Y, {\mathfrak{g}}_2]$ is $\sigma$-stable.  
Since ${\mathfrak{Z}}_{\mathfrak{g}}(X)\supset {\mathfrak{g}}_1 \oplus {\mathfrak{g}}_2$
 by Lemma \ref{lem:22040704} (2), 
 the decomposition \eqref{eqn:g012} yields
\begin{align*}
   {\mathfrak{g}}
=& \sigma({\mathfrak{Z}}_{\mathfrak{g}}(X)) + {\mathfrak{g}}_0 + {\mathfrak{Z}}_{\mathfrak{g}}(X)
\\
=&{\mathfrak{g}}^{\sigma} + {\mathfrak{g}}_0 +{\mathfrak{Z}}_{\mathfrak{g}}(X).
\end{align*}
By Lemma \ref{lem:22040704} (4), 
 the first equality of Lemma \ref{lem:22040623} follows 
 because ${\mathfrak{Z}}_{\mathfrak{g}}(X) \supset {\mathfrak{Z}}_{\mathfrak{g}}({\mathfrak{sl}}_2^X)$.  
Since ${\mathfrak{g}}_0 + {\mathfrak{Z}}_{\mathfrak{g}}(X)
=[X, {\mathfrak{g}}_{-2}] + {\mathfrak{Z}}_{\mathfrak{g}}(X)$
 by the irreducible decomposition \eqref{eqn:gsl2}, 
 the second equality holds.  
\end{proof}

We now give a proof of Theorems \ref{thm:22040624} and \ref{thm:22100105}.  
For a complex simple Lie algebra ${\mathfrak{g}}$, 
 the statement is reduced to Fact \ref{fact:22030921}
 as we have seen in the proof of Theorem \ref{thm:Joseph}.  
So it suffices to treat the case
 where ${\mathfrak{g}}$ is absolutely simple.

\begin{proof}[Proof of Theorem \ref{thm:22040624}]
We set $L:=M \exp ({\mathbb{R}}A_{\mu})$, 
 and 
\[
   S:=\operatorname{Ad}(L)X
\subset {\mathbb{O}}_{\operatorname{min}, {\mathbb{R}}}
=\operatorname{Ad}(G)X.  
\]
Then $\operatorname{Ad}(G^{\sigma}) S$ is open in ${\mathbb{O}}_{\operatorname{min}, {\mathbb{R}}}$
 by Lemma \ref{lem:22040623}.  
We now verify the condition of Lemma \ref{lem:1.7}:
\begin{equation}
\label{eqn:hZX}
({\mathfrak{g}}^{\sigma}+{\mathfrak{Z}}_{\mathfrak{g}}(W))^{\perp}\subset 
[W, {\mathfrak{g}}^{\sigma}]
\qquad
\text{for all $W \in S$.}
\end{equation}
Since $S \subset {\mathfrak{g}}({\mathfrak{a}};\mu)$, 
 it suffices
 to show \eqref{eqn:hZX}
 for $W=X$.  
By the second equality in Lemma \ref{lem:22040623}, 
 one has 
$
({\mathfrak{g}}^{\sigma}+{\mathfrak{Z}}_{\mathfrak{g}}(X))^{\perp}
\subset
[X, {\mathfrak{g}}_{-2}].  
$
Since 
$\sigma({\mathfrak{g}}_{-2})={\mathfrak{g}}_2$ is abelian, 
 one has
\[
[X, {\mathfrak{g}}_{-2}]
=\{[X, V+\sigma(V)]: V \in {\mathfrak{g}}_{-2}\}
\subset [X, {\mathfrak{g}}^{\sigma}].  
\]
Thus \eqref{eqn:hZX} is shown.  
\end{proof}

\begin{proof}
[Proof of Theorem \ref{thm:22100105}]
The coadjoint orbit
 ${\mathbb{O}}_{\operatorname{min}, {\mathbb{R}}}^{\mathbb{C}}$
 is of the form 
$
   {\mathbb{O}}_{\operatorname{min}, {\mathbb{R}}}^{\mathbb{C}}
   =
   \operatorname{Ad}(G_{\mathbb{C}})X
\simeq
G_{\mathbb{C}}/(G_{\mathbb{C}})_X
$
 for any non-zero $X \in {\mathfrak{g}}({\mathfrak{a}};\mu)$
 via the identification ${\mathfrak{g}}_{\mathbb{C}}^{\ast} \simeq {\mathfrak{g}}_{\mathbb{C}}$. 
We take $Y \in {\mathfrak{g}}({\mathfrak{a}};-\mu)$
 such that $\{X, A_{\mu}, Y \}$ forms an ${\mathfrak{s l}}_2$-triple
 as before.   
Since ${\mathbb{O}}_{\operatorname{min}, {\mathbb{R}}}^{\mathbb{C}}$ contains $Y$, 
 one can also write as 
 ${\mathbb{O}}_{\operatorname{min}, {\mathbb{R}}}^{\mathbb{C}} = \operatorname{Ad}(G_{\mathbb{C}})Y \simeq G_{\mathbb{C}}/(G_{\mathbb{C}})_Y$.  
Then Lemma \ref{lem:22040704} implies that 
\begin{equation}
\label{eqn:Zglmd}
  {\mathfrak{g}}
  =
  {\mathfrak{Z}}_{\mathfrak{g}}(Y)
  +({\mathfrak{m}}+{\mathbb{R}}A_{\mu})
  +{\mathfrak{Z}}_{\mathfrak{g}}(X).  
\end{equation}

We take any nonzero $Y' \in {\mathfrak{g}}({\mathfrak{a}};-\mu)$.  
We claim
\begin{equation}
\label{eqn:XYprime}
(\operatorname{diag}{\mathfrak{g}}
  + {\mathfrak{Z}}_{{\mathfrak{g}} \oplus {\mathfrak{g}}} (X,Y'))^{\perp}
 \subset 
[(X,Y'), \operatorname{diag} {\mathfrak{g}}].  
\end{equation}
In fact, 
 by using the decomposition \eqref{eqn:g012} via the ${\mathfrak{s l}}_2$-triple
 $\{X, A_{\mu}, Y\}$, 
 one has 
 ${\mathfrak{Z}}_{\mathfrak{g}}(Y') \supset {\mathfrak{g}}_{-1} \oplus {\mathfrak{g}}_{-2}$, 
hence
\[
 ({\mathfrak{Z}}_{{\mathfrak{g}}}(X) + {\mathfrak{Z}}_{{\mathfrak{g}}}(Y'))^{\perp}
\subset
 ({\mathfrak{Z}}_{{\mathfrak{g}}}(X) \oplus {\mathfrak{g}}_{-1} \oplus {\mathfrak{g}}_{-2})^{\perp}
=[X, {\mathfrak{g}}_{-2}]
=[X, {\mathfrak{g}}({\mathfrak{a}};-\mu)]
\]
 by the representation theory of ${\mathfrak{s l}}_2^X$.  
Switching the role of $X$ and $Y'$, 
 one sees
\[
({\mathfrak{Z}}_{{\mathfrak{g}}}(X) + {\mathfrak{Z}}_{{\mathfrak{g}}}(Y'))^{\perp}\subset [Y', {\mathfrak{g}}({\mathfrak{a}};\mu)].  
\]
Hence the left-hand side of \eqref{eqn:XYprime} is contained
 in 
\[
\{(Z,-Z): Z \in [X,{\mathfrak{g}}({\mathfrak{a}};-\mu)] \cap [Y', {\mathfrak{g}}({\mathfrak{a}};\mu)]\}, 
\]
 which is a subspace 
 of $[(X,Y'), \operatorname{diag} {\mathfrak{g}}]$
 because both ${\mathfrak{g}}({\mathfrak{a}};\mu)$ and ${\mathfrak{g}}({\mathfrak{a}};-\mu)$ are abelian.  
Hence \eqref{eqn:XYprime} is shown.

We set $L_{\mathbb{C}}:=M_{\mathbb{C}} \exp({\mathbb{C}}A_{\mu})$
 and $S:=\{(\operatorname{Ad}(\ell)X, \operatorname{Ad}(\ell^{-1})Y):\ell \in L_{\mathbb{C}}\}$.  
By \eqref{eqn:Zglmd}, 
 $\operatorname{diag}(G_{\mathbb{C}}) S$ is open dense
 in ${\mathbb{O}}_{\operatorname{min}, {\mathbb{R}}}^{\mathbb{C}} \times {\mathbb{O}}_{\operatorname{min}, {\mathbb{R}}}^{\mathbb{C}}$ 
 in light of the identification 
 $\operatorname{diag} (G_{\mathbb{C}}) \backslash (G_{\mathbb{C}} \times G_{\mathbb{C}}) \simeq  G_{\mathbb{C}}$, 
 $(x,y) \mapsto x^{-1} y$.  

Similarly to \eqref{eqn:XYprime}, 
 one obtains the following inclusion:
\begin{equation*}
   (\operatorname{diag}({\mathfrak{g}}_{\mathbb{C}})
    +
    {\mathfrak{Z}}_{{\mathfrak{g}}_{\mathbb{C}}\oplus {\mathfrak{g}}_{\mathbb{C}}}
    (\operatorname{Ad}(\ell)X,\operatorname{Ad}(\ell^{-1})Y))^{\perp}
   \subset
   [(\operatorname{Ad}(\ell)X,\operatorname{Ad}(\ell^{-1})Y), \operatorname{diag}({\mathfrak{g}}_{\mathbb{C}})]
\end{equation*}
for any $\ell \in L_{\mathbb{C}}$.  
Thus Theorem \ref{thm:22100105} follows from Lemma \ref{lem:1.7}.  
\end{proof}

\subsection{Singular representations
 of $Sp(p,q)$ and $F_{4(-20)}$}
\label{subsec:smallf4}
In this section, 
 we verify 
 the existence of $\Pi \in \operatorname{Irr}(G)$
 satisfying
 $\operatorname{DIM}(\Pi)=m({\mathfrak{g}})$
 for ${\mathfrak{g}}={\mathfrak{s p}}(p,q)$
 or ${\mathfrak{f}}_{4(-20)}$.  
Actually, 
 one can take $\Pi$ 
 to be the globalization 
 of Zuckerman's module $A_{\mathfrak{q}}(\lambda)$, 
 a cohomological parabolic induction 
 for some $\theta$-stable parabolic subalgebra
 ${\mathfrak{q}}$ in ${\mathfrak{g}}_{\mathbb{C}}$.

In what follows, 
 we write ${\mathfrak{q}}={\mathfrak{l}}_{\mathbb{C}}+{\mathfrak{u}}$
 for the Levi decomposition
 of a $\theta$-stable parabolic subalgebra ${\mathfrak{q}}$
 of ${\mathfrak{g}}_{\mathbb{C}}={\mathfrak{k}}_{\mathbb{C}}+{\mathfrak{p}}_{\mathbb{C}}$, 
 where ${\mathfrak{l}}_{\mathbb{C}}$ is the complexified Lie algebra
 of $L=N_G({\mathfrak{q}})$, 
 the normalizer of ${\mathfrak{q}}$ in $G$.  
Then the Gelfand--Kirillov dimension of $A_{\mathfrak{q}}(\lambda)$
 is the complex dimension 
 of $\operatorname{Ad}(K_{\mathbb{C}})({\mathfrak{u}} \cap {\mathfrak{p}}_{\mathbb{C}})$, 
 see {\it{e.g., }} \cite{xkInvent98}.

\begin{lemma}
\label{lem:22092502}
Let $G=Sp(p,q)$, 
 and ${\mathfrak{q}}$ be a $\theta$-stable parabolic subalgebra
 with $L \simeq {\mathbb{T}} \times Sp(p-1,q)$
 or $Sp(p,q-1) \times {\mathbb{T}}$.  
Then $\operatorname{DIM}(A_{\mathfrak{q}}(\lambda))=2(p+q)-1$.  
\end{lemma}

\begin{proof}
We identify ${\mathfrak{p}}_{\mathbb{C}} \simeq M(2p,2q;{\mathbb{C}})$.  
Then $\operatorname{Ad}(K_{\mathbb{C}})({\mathfrak{u}} \cap {\mathfrak{p}}_{\mathbb{C}})$
 is contained in the variety
 of rank one matrices for the above parabolic subalgebra ${\mathfrak{q}}$, 
 which is of complex dimension
 $2(p+q)-1$.  
Hence $\operatorname{DIM}(A_{\mathfrak{q}}(\lambda)) \le 2(p+q)-1$.  
Since $\operatorname{DIM}(A_{\mathfrak{q}}(\lambda)) \ge 
 m({\mathfrak{g}}) = 2(p+q)-1$, 
 we obtain the desired equality.  
\end{proof}

\begin{lemma}
\label{lem:22093002}
Let $G=F_{4(-20)}$, 
 and ${\mathfrak{q}}$ be one of $\theta$-stable parabolic subalgebras of ${\mathfrak{g}}_{\mathbb{C}}$
 in {\rm{\cite[Table C.4]{decoAq}}}.  
Then $\operatorname{DIM}(A_{\mathfrak{q}}(\lambda))=11$.  
\end{lemma}

\begin{proof}
For ${\mathfrak{g}}={\mathfrak{f}}_{4(-20)}$, 
 there are three real nilpotent coadjoint orbits, 
 and their dimensions are 0, 22, 30.  
This implies
 that $\operatorname{DIM}(\Pi)\in \{0, 11, 15\}$
for any $\Pi \in \operatorname{Irr}(G)$.

The asymptotic $K$-support
 (Section \ref{subsec:adm})
 of $A_{\mathfrak{q}}(\lambda)$ 
 has the following upper estimate
 $\operatorname{AS}_K(A_{\mathfrak{q}}(\lambda))
 \subset {\mathbb{R}}_+ \langle {\mathfrak{u}} \cap {\mathfrak{p}}_{\mathbb{C}} \rangle$, 
 see \cite[Ex.\ 3.2]{xkAnn98}
 and the notation therein.

As we saw in \cite{decoAq}, 
 the asymptotic $K$-support $\operatorname{AS}_K(A_{\mathfrak{q}}(\lambda))$
 for the parabolic subalgebra ${\mathfrak{q}}$ 
 under consideration is strictly smaller
 than that of a principal series representation of $G$.  
In turn, 
 by \cite[Prop.\ 2.6]{K21Kostant}, 
 we conclude that $\operatorname{DIM}(A_{\mathfrak{q}}(\lambda))<m({\mathfrak{g}})=15$.  
Hence $\operatorname{DIM}(A_{\mathfrak{q}}(\lambda))=11$.  
\end{proof}

\begin{remark}
\label{rem:22092511}
The Gelfand--Kirillov dimensions
 of irreducible representations
 are known
 for the group of real rank one.  
In particular, 
 one may observe from \cite[Fig.\ 8.16]{collingwood85}
 that 
$
  \operatorname{DIM} \colon \operatorname{Irr}(F_{4(-20)}) \to \{0,11,15\}
$
is surjective.  
\end{remark}

\begin{proof}
[Proof of Theorems \ref{thm:mbrest} and \ref{thm:tensor}]
We have shown 
 at the end of Section \ref{sec:min}
 that the remaining cases are ${\mathfrak{g}}={\mathfrak{sp}}(p,q)$
 or ${\mathfrak{f}}_{4(-20)}$.  
For these Lie algebras, 
 the assumption $\sigma \mu=-\mu$ in Theorem \ref{thm:mgbdd}
 is satisfied
 (see Example \ref{ex:6.2}).  
On the other hand, 
 we also have verified
 that there exists $\Pi \in \operatorname{Irr}(G)$
 with $\operatorname{DIM}(\Pi)=m({\mathfrak{g}})$
 for these Lie algebras ${\mathfrak{g}}$.  
Hence Theorem \ref{thm:mgbdd} covers the remaining cases, 
 and completes the proof of Theorems \ref{thm:mbrest} and \ref{thm:tensor}.  
\end{proof}

\vskip 1pc
{\bf{$\langle$Acknowledgement$\rangle$}}
This work was partially supported
 by Grant-in-Aid for Scientific Research (A) (18H03669), 
JSPS.  
The author would like to thank the anonymous referee
 for making helpful comments.  
Thanks are also due to Toshihisa Kubo 
 for reading carefully the first draft of this paper.


\begin{thebibliography}{00}
\small
\bibitem{vdB}
E.\ P.\ van den Ban, 
{\textit{Invariant differential operators
 on a semisimple symmetric space
 and finite multiplicities 
 in a Plancherel formula}}, 
 Ark.\ Mat., {\bf{25}}
 (1987)
175--187.  

\bibitem{Be57}
M.\ Berger, 
{\textit{Les espaces sym{\'e}triques non-compacts}}, 
Ann.\ Sci.\ {\'E}cole Norm.\ Sup. 
(3) {\bf{74}} (1957),  85--177.  

\bibitem{BZ}
B.\ Binegar, R.\ Zierau, 
{\textit{A singular representation of $E_6$}}, 
Trans.\ Amer.\ Math.\ Soc., 
{\bf{341}} (1994), 
771--785.  

\bibitem{BB82}
W.\ Borho, J.\ Brylinski, 
{\textit{
Differential operators on homogeneous spaces I: 
irreducibility of the associated variety}}, 
Invent.\ Math., 
{\bf{69}} (1982), 
437--476,  

\bibitem{B98}
R.\ Brylinski, 
{\textit{Geometric quantization of real minimal nilpotent orbits}}, 
Differential Geom.\ Appl.\ {\bf{9}} (1998), 5--58.

\bibitem{Clare12}
P.\ Clare, 
{\textit{On the degenerate principal series of complex symplectic groups}}, 
 J.\ Funct.\ Anal. {\bf{262}} (2012), 
4160--4180.  

\bibitem{Clerc}
J.-L.\ Clerc, 
{\textit{Singular conformally invariant trilinear forms, I:
The multiplicity one theorem}}, 
Transform.\ Groups {\bf{21}} (2016), 619--652.  

\bibitem{collingwood85}
D.\ H.\ Collingwood, 
{\textit{Representations of Rank One Lie Groups}},
 Research Notes in Mathematics, 
 {\bf{137}}, 
 Pitman, 1985, 
vii+244pp.  

\bibitem{C93}
D.\ H.\ Collingwood, W.\ M.\ McGovern, 
{\textit{Nilpotent Orbits in Semisimple Lie Algebras}}, 
Van Nostrand Reinhold Mathematics Series, Van Nostrand 
Reinhold Co., New York, 1993, 
xiv+186pp.

\bibitem{xdgv}
M.\ Duflo, E.\ Galina, J.\ A.\ Vargas, 
{\textit{Square integrable representations of reductive Lie groups 
with admissible restriction to $SL_2({\mathbb R})$}},
J.~Lie Theory, {\bf{27}}, (2017), 1033--1056.  

\bibitem{GS05}
W.\ T.\ Gan, G.\ Savin, 
{\textit{On minimal representations definitions and properties}}, 
Represent.\ Theory {\bf{9}} (2005), 46--93. 

\bibitem{xhnoo}
X.~He, 
K.~Nishiyama, 
H.~Ochiai, 
Y.~Oshima, 
{\textit{On orbits in double flag varieties for symmetric pairs}}, 
\newblock Transform.~Groups, {\bf{18}} (2013), 1091--1136.  

\bibitem{huwu90}
A.\ T.\ Huckleberry, T.\ Wurzbacher, 
{\textit{Multiplicity-free complex manifolds}}, 
 Math.\ Ann.\ {\bf{286}} (1990), 261--280.  

\bibitem{J79}
A.\ Joseph, 
{\textit{The minimal orbit in a simple Lie algebra
 and its associated maximal ideal}}, 
 Ann.\ Sci.\ Ecole Norm.\ Sup.
 {\bf{9}} (1976), 1--29.  

\bibitem{Joseph85}
A.\ Joseph, 
{\textit{On the associated variety of a primitive ideal}}, 
 J.\ Algebra, 
{\bf{93}} (1985), 509--523.  

\bibitem{kaneyuki85}
S.\ Kaneyuki, M.\ Kozai, 
{\textit{Paracomplex structures and affine symmetric spaces}}, 
Tokyo J.\ Math., 
{\bf{8}} (1985), 81--98.  

\bibitem{Ki}
M.\ Kitagawa,
{\textit{Uniformly bounded multiplicities, polynomial identities and 
 coisotropic actions}}, 
arXiv:2109.05555v2. 

\bibitem{xkInvent94}
T.~Kobayashi, 
{\textit{Discrete decomposability of the restriction of
             $A_{\frak q}(\lambda)$
            with respect to reductive subgroups and its applications}}, 
Invent. Math.,
{\bf{117}} 
(1994), 
\href{http://dx.doi.org/10.1007/BF01232239}
{181--205}.  

\bibitem{K95}
T.\ Kobayashi, 
{\textit{Introduction to harmonic analysis on real spherical 
homogeneous spaces}}, 
In: Proceedings of the 3rd Summer School on Number 
Theory \lq\lq{Homogeneous Spaces and Automorphic Forms}\rq\rq\
 in Nagano, (ed.\ F.\ Sato), 1995,  22--41.

\bibitem
{xkAnn98}
T.~Kobayashi, 
{\textit{Discrete decomposability of the restriction of
             $A_{\frak q}(\lambda)$
            with respect to reductive subgroups {\rm{II}}---micro-local analysis and asymptotic $K$-support}},  
Ann. of Math., 
{\bf {147}} 
(1998), 
\href{http://dx.doi.org/10.2307/120963}
{709--729}.  

\bibitem{xkInvent98}
T.~Kobayashi, 
{\textit{Discrete decomposability of the restriction of
             $A_{\frak q}(\lambda)$
            with respect to reductive subgroups {\rm{III}}---restriction of Harish-Chandra modules
 and associated varieties}}, 
Invent. Math., {\bf{131}} (1998), 
\href{http://dx.doi.org/10.1007/s002220050203}
{229--256}.  

\bibitem{JFA98}
T.\ Kobayashi, 
{\textit{Discrete series representations for the orbit spaces
 arising from two involutions of real reductive Lie groups}}, 
J.\ Funct.\ Anal. {\bf{152}} (1998), 
100--135.  

\bibitem{xrims40}
T.\ Kobayashi, 
{\textit{Multiplicity-free representations and visible actions
on complex manifolds}}, 
\newblock Publ.~Res.~Inst.~Math.~Sci.,  
{\textbf{41}} (2005),
\href{http://dx.doi.org/10.2977/prims/1145475221}
{497--549},
special issue commemorating the fortieth anniversary of the founding of RIMS.

\bibitem{visibleGK}
T.\ Kobayashi, 
{\textit{Visible actions on symmetric spaces}}, 
Transform.\ Groups, 
{\bf{12}}
(2007), 671--694.  

\bibitem{mf-korea}
T.\ Kobayashi, 
{\textit{Multiplicity-free theorems 
 of the restrictions of unitary highest weight modules
 with respect to reductive symmetric pairs}}, 
Progr.\ Math., {\bf{255}}
(2007), 45--109.  

\bibitem{K11Zuckerman}
T.\ Kobayashi, 
\textit{Branching problems of Zuckerman derived functor modules}, 
In: Representation Theory and Mathematical Physics
---In Honor of G.~Zuckerman,
(eds.\ J.\ Adams, B.\ Lian, and S.\ Sahi), 
Contemp. Math., 
{\bf{557}}, 
\href{http://www.ams.org/books/conm/557/11024}
{23--40}, 
Amer. Math. Soc., Providence, RI, 2011.


\bibitem{K14}
T.\ Kobayashi, 
{\textit{Shintani functions, real spherical manifolds, and 
symmetry breaking operators}}, 
Dev.\ Math., {\bf{37}} (2014), 127--159.  

\bibitem{xKVogan2015}
T.~Kobayashi, 
{\textit{A program for branching problems in the representation 
theory of real reductive groups}}, 
In: Representations of Reductive 
Groups---In Honor of the 60th Birthday of David A.\ Vogan,
Jr., 
Progr. Math., 
{\bf{312}}, 
Birkh{\"a}user, 
2015,
\href{http://dx.doi.org/10.1007/978-3-319-23443-4_10}
{277--322}.  

\bibitem{K21}
T.~Kobayashi, 
{\textit{Branching laws of unitary representations
 associated to minimal elliptic orbits 
 for indefinite orthogonal group $O(p,q)$}}, 
Adv.~Math., {\bf{388}} (2021), 
Paper No. 107862, 38pp.  

\bibitem{K21Kostant}
T.\ Kobayashi, 
{\textit{Admissible restrictions of irreducible representations
 of reductive Lie groups: 
 Symplectic geometry and discrete decomposability}}, 
Pure Appl.\ Math.\ Q.\
 {\bf{17}}
 (2021), 
1321--1343, 
Special volume in memory of Bertram Kostant.  

\bibitem{K22PJA}
T.\ Kobayashi, 
{\textit{Multiplicity in restricting small representations}}, 
Proc.\ Japan Acad.\ Ser.\ A, {\bf{98}} (2022), 19--24.  

\bibitem{K22}
T.\ Kobayashi, 
{\textit{Bounded multiplicity theorems for induction and  restriction}}. 
J.\ Lie Theory {\bf{32}} (2022), 197--238.

\bibitem{tkVarnaMin}
T.\ Kobayashi, 
{\textit{Multiplicity in restricting minimal representations}}, 
Springer Proc.\ Math. Stat.\ 
{\bf{396}} (2022), 1--18.  
Available also at arXiv:2204.05079.  

\bibitem{xKMt}
T.\ Kobayashi, T. Matsuki, 
{\textit{Classification of finite-multiplicity symmetric pairs}},
Transform. Groups, 
 {\bf{19}} 
(2014), 
\href{http://dx.doi.org/10.1007/s00031-014-9265-x }
{457--493}, 
Special issue in honor of Dynkin
 for his 90th birthday. 

\bibitem{KOr}
T.\ Kobayashi, B.\ \O rsted, 
{\textit{Analysis on the minimal representation of $O(p,q)$}}. II. 
Branching laws. Adv. Math. {\bf{180}} (2003), 513--550. 

\bibitem{KOP}
T.\  Kobayashi, B.\ \O rsted, M.\ Pevzner,
{\textit{Geometric  analysis on small unitary representations of $GL(N,{\mathbb{R}})$}}, 
J.\ Funct.\ Anal.\ {\bf{260}} (2011), 1682--1720. 

\bibitem{KOPU09}
T.\ Kobayashi, B.\ {\O}rsted, M.\ Pevzner, A.\ Unterberger, 
{\textit{Composition formulas
 in the Weyl calculus}}, 
 J.\ Funct.\ Anal., 
{\bf{257}} (2009), 948--991.  

\bibitem{xktoshima}
T.~Kobayashi, T.~Oshima, 
{\textit{Finite multiplicity theorems for induction and restriction}}, 
Adv. Math., \textbf{248} (2013), 
\href{http://dx.doi.org/10.1016/j.jfa.2010.12.008}
{921--944}. 

\bibitem{decoAq}
T.\ Kobayashi, Y.\ Oshima, 
\textit{Classification of discretely decomposable 
$A_{\mathfrak {q}}(\lambda)$
 with respect to reductive symmetric pairs}, 
Adv. Math., 
{\bf{231}}
 (2012), 
{2013--2047}.

\bibitem{KO15}
T.\ Kobayashi, Y.\ Oshima, 
{\textit{Classification of symmetric pairs with discretely decomposable 
 restrictions of $(\mathfrak g,K)$-modules}}. 
  J.\ Reine Angew.\ Math.\ {\bf{703}} (2015), 201--223.

\bibitem{KS15}
T.\ Kobayashi, B.\ Speh,
{\textit{Symmetry Breaking
 for Representations of Rank One
 Orthogonal Groups}}, 
(2015), 
 Mem.\ Amer.\ Math.\ Soc. 
 {\bf{238}} 
\href{http://dx.doi.org/10.1090/memo/1126}
{no.1126}, 
 118 pages.  

\bibitem{xksbonvec}
T.~Kobayashi, B.~Speh, 
{\textit{Symmetry Breaking for Representations
 of Rank One Orthogonal Groups}}, 
Part II, 
\href{https://arxiv.org/abs/1801.00158}
Lecture Notes in Math., {\bf{2234}} 
Springer, 2018.  
xv$+$342 pages.  

\bibitem{xkramer}
M.~Kr{\"a}mer, 
{\textit{Multiplicity free subgroups
 of compact connected Lie groups}},
 Arch. Math. (Basel)
 {\bf{27}}, (1976), 28--36.  

\bibitem{Li94}
P.~Littelmann, 
{\textit{On spherical double cones}}, 
\newblock J.~Algebra, {\bf{166}} (1994), 142--157.  

\bibitem{O15}
T.\ Okuda, 
{\textit{Smallest complex nilpotent orbits with real points}}. 
J.\ Lie Theory {\bf{25}} (2015), 507--533. 

\bibitem{OK22}
T.\ Okuda, 
Some remarks on real minimal nilpotent orbits
 and symmetric pairs, 
 in preparation.  

\bibitem{SunZhu}
B.~Sun, C.-B.~Zhu,
\textit{Multiplicity one theorems:
the Archimedean case},
Ann. of Math. (2), 
{\bf{175}} 
(2012), 
23--44.  


\bibitem{Ta}
H.\ Tamori, 
{\textit{Classification of minimal representations of real simple Lie 
groups}}. 
Math.\ Z.\ {\bf{292}} (2019), 387--402.

\bibitem{xtanaka12}
Y.~Tanaka, 
{\textit{Classification of visible actions on flag varieties}}, 
\newblock Proc.~Japan Acad.~Ser.~A Math.~Sci., {\bf{88}} (2012), 91--96.  

\bibitem{Tu}
T.\ Tauchi, 
{\textit{A generalization of the Kobayashi--Oshima
 uniformly bounded multiplicity theorem}}, 
Internat.\ J.\ Math., {\bf{32}} (2021), 
2150107.

\bibitem{WaI}
N.\ R.\ Wallach,
{\textit{Real reductive groups}}. I, II, 
Pure Appl.\ Math. 
{\bf{132}} 
 Academic Press, Inc., Boston, MA, 1988;
{\bf{132}}-II, ibid, 1992. 
\end{thebibliography}
\end{document}